\documentclass{article}
\usepackage{graphicx,subcaption}
\usepackage{mathtools}
\mathtoolsset{showonlyrefs}
\usepackage{amssymb}
\usepackage{amsmath}
\usepackage{amsthm}
\usepackage{xcolor,soul}
\usepackage{fancyhdr}
\usepackage{epstopdf}
\usepackage[colorlinks=true,citecolor=magenta,linkcolor=blue]{hyperref}

\makeatletter
\let\@fnsymbol\@arabic
\makeatother

\usepackage[]{geometry}
\newcommand{\correz}[1]{\textcolor{black}{#1}}

\newtheorem{theorem}{Theorem}
\newtheorem{lemma}{Lemma}
\newtheorem{proposition}{Proposition}
\newtheorem{remark}{Remark}
\newtheorem{definition}{Definition}

\usepackage{tikz}

\usetikzlibrary{positioning}
\usetikzlibrary{arrows}
\usetikzlibrary{arrows.meta}
\usetikzlibrary{calc}
\usetikzlibrary{shapes}

\graphicspath{ {./figures/} }

\title{A geometric analysis of the Bazykin-Berezovskaya predator-prey model with Allee effect in an economic framework}
\author{Jacopo Borsotti\thanks{Department of Mathematical, Physical and Computer Sciences, University of Parma, Parco Area delle Scienze 53/A, 43124 Parma, Italy, {\tt jacopo.borsotti@unipr.it}}\, and Mattia Sensi\thanks{Dipartimento di Matematica, Università degli Studi di Trento, Via Sommarive 14, 38123 Povo (Trento), Italy, {\tt mattia.sensi@unitn.it}}}
\date{}

\begin{document}

\maketitle

\begin{abstract}
We study a fast-slow version of the Bazykin-Berezovskaya predator-prey model with Allee effect evolving on two timescales, through the lenses of Geometric Singular Perturbation Theory (GSPT). The system we consider is in non-standard form. We completely characterize its dynamics, providing explicit threshold quantities to distinguish between a rich variety of possible asymptotic behaviors. Moreover, we propose numerical results to illustrate our findings. Lastly, we comment on the real-world interpretation of these results, in an economic framework and in the context of predator-prey models.
\end{abstract}

\section{Introduction} \label{section1}

Systems of Ordinary Differential Equations (ODEs) have historically provided a convenient and often analytically tractable approach to the mathematical modeling of natural phenomena. In particular, relevant to this work, many different variations of planar systems have been proposed over the years to describe the interaction of predator and prey populations \cite{4,dawes2013derivation}. 

Often, natural phenomena are a result of interacting mechanisms which evolve on widely different timescales: we refer in particular to \cite{40} for a wide array of examples, ranging from engineering to pattern formation, from lasers to celestial mechanics. A relatively recent technique to analyze these systems is the so-called Geometric Singular Perturbation Theory (GSPT), stemming from the seminal works of Neil Fenichel \cite{23}, and complemented with more advanced tools such as the \emph{entry-exit function} \cite{34,35}.

This approach has been applied to a variety of modeling scenarios: epidemics \cite{jardon2021geometric,jardon2021geometric2,bulai2024geometric,kaklamanos2024geometric,SIRS}, neuroscience \cite{rodrigues2016time,kaklamanos2023geometric,sensi2023slow}, ecology \cite{iuorio2021influence,grifo2025far}, chemistry \cite{gucwa2009geometric,kuehn2015multiscale,taghvafard2021geometric}, \dots \, In particular, one common underlying fact for many of these models, if we abstract from the case-specific biological details, is that it takes much longer for a resource to develop than it takes to consume it. This is the case for ``replenishing'' the susceptible individuals pool in an epidemic model with prolonged or permanent immunity upon recovery, or for prey reproduction in a predator-prey setting, or for resource growth in a resource-production scenario.

This work focuses on a planar model originally introduced in \cite{bazykin1979allee} (see also \cite[Section 3.5.5]{4}) to describe the interaction between a naturally growing resource (governed by an Allee effect \cite{allee}) and an agent using said resource for production (e.g., trees and woodcutting). Naturally, the growth of the resource is much slower than its harvesting and consumption. For example, this is the case of deforestation \cite{forest} or oil extraction \cite{SORRELL20105290}. This motivates us to introduce a distinction in the order of magnitude of the parameters involved in our system. In particular, the system under study is a fast-slow system in \emph{non-standard form} \cite{wechselberger2020geometric}, meaning that neither of the variables of the system is \emph{globally} fast or slow, but the separation of timescales is evident close to a manifold representing the absence of production. We exploit this separation of timescales, encapsulated in a small parameter $0<\varepsilon \ll 1$, to fully characterize the possible asymptotic behaviors of the system, depending on explicit relations between the $\mathcal{O}(1)$ parameters of the model. The analytical results allow us to distinguish between several economic scenarios, each describing a different evolution of resources and production, and to identify those in which long-term production is maximized while ensuring that resources are not depleted. 

We remark that our model can also be interpreted as a predator-prey system in which the birth rate of the prey population is significantly lower than the other rates governing species interactions and predator mortality. Such timescale separation is a well-documented phenomenon in predator-prey dynamics. Notable examples include interactions between ungulates and their predators \cite{ungulati,peterson2003temporal}, algae and rotifers \cite{alghe}, and deer and lynx populations \cite{andren2024numerical}. \correz{We refer the interested reader to three recent papers which study existence, number, and stability of limit cycles (specifically, relaxation oscillations) in planar slow-fast systems of ODEs describing predator-prey interactions \cite{Huzak2018,hsu2019number,ai2024relaxation}.}

The remainder of this manuscript is structured as follows. In Section \ref{section2}, we introduce the model under study, explaining its economical interpretation, listing our assumptions, and recalling known results on its equilibria and their stability. In Section \ref{section3}, we provide our multiple timescale analysis of the model, after a brief recollection of the basis of GSPT and of the entry-exit function. In Section \ref{section4}, we illustrate our analytical results through a selection of numerical simulations and bifurcation analysis. In Section \ref{section5}, we provide economical and biological interpretations of the most interesting possible asymptotic behaviors of the model. We conclude in Section \ref{section6}.

\section{The model} \label{section2}

In this section, we introduce an ODE model which describes the evolution of the quantity of resources available $R$ and the production $P$ of a society. Our goal is to develop a model capable of reproducing an economy where the natural growth rate of resources is significantly lower than the other parameters governing the system. This implies that resources are consumed at a much higher rate than they are recreated. Hence, we want to understand under what conditions all resources fade away and when, on the other hand, the system reaches stability (which could be either a fixed point or a limit cycle).

Consider the Bazykin-Berezovskaya model \cite{4,bazykin1979allee}, which is often used to simulate predator-prey dynamics, 
\begin{align} \label{eq1}
\left\{
\begin{aligned}
    \frac{dR(t)}{dt} &= \varepsilon n R(t) (R(t)-L) (M - R(t)) - e R(t) P(t), \\ 
    \frac{dP(t)}{dt} &= b R(t) P(t) - d P(t). \\ 
\end{aligned}
\right.
\end{align}
The parameters of the system are the following: 
\begin{itemize}
    \item $\varepsilon n>0$ represents the natural rate of growth of the resources; 
    \item $M>0$ represents the maximum amount of resources that the environment is able to produce; 
    \item $0<L<M$ is the resource extinction threshold; 
    \item $e>0$ represents the resource extraction rate; 
    \item $b>0$ represents the resource utilization efficiency; 
    \item $d>0$ is the production decline rate; 
    \item $0<\varepsilon \ll 1$ is a small parameter highlighting the fact that the natural growth rate of resources is significantly lower than the other parameters governing the system, in particular $\varepsilon \ll n,M,l,e,p,d$. 
\end{itemize}
The growth of the resources in absence of production is modeled with a logistic growth adapted with an Allee effect. The consume of the resources and the consequent increase in production are modeled as a Lotka-Volterra functional. The production is assumed to decrease exponentially when no resources are available. Note that if $bM<d$, then the production $P$ can only decrease. Model \eqref{eq1} can also be interpreted as a predator-prey scenario where the rate of growth of the preys is much smaller than the other rates governing the system.  

Applying the change of variables 
$$R=Mu, \quad P=\dfrac{nM^2}{e}v, \quad t=\dfrac{\tilde{t}}{nM^2},$$ and dropping the dependence on the time variable for ease of notation, system \eqref{eq1} can be rewritten as 
\begin{align} \label{eq2}
\left\{
\begin{aligned}
    \dot{u} &= \varepsilon u (u-l) (1-u)  - uv, \\ 
    \dot{v} &= -\gamma v(m-u), \\ 
\end{aligned}
\right.
\end{align}
where 
$$l=L/M\in (0,1), \quad \gamma=b/(nM)>0, \quad m=d/(bM)>0.$$ 
System \eqref{eq2} describes the evolution of the new variables $u$ and $v$, the overdot $^\cdot$ indicates the derivative with respect to the (new) time variable (which we will keep denoting as $t$). Note that $bM<d$ is equivalent to $m>1$. System \eqref{eq2} evolves in the economically relevant region 
\begin{equation} \label{eq3}
    \Delta \coloneqq \{(u,v) \in \mathbb{R}^2 \colon \; u \in [0,1], \; v \ge 0\}, 
\end{equation}
indeed $\dot{v}|_{v=0}=0$, $\dot{u}|_{u=0}=0$, and $\dot{u}|_{u=1}=-v\le 0$. Moreover, it evolves on two different timescales: the fast timescale $t$, and the slow timescale $\tau=\varepsilon t$ (this two-timescale structure will be made explicit in Section \ref{section3}). 

\subsection{Equilibria, stability, and bifurcation analysis}  \label{section2_2}

The equilibria and the bifurcations of system \eqref{eq2} are known \cite{4,3}. Here we provide a summary of the results which will be useful for our analysis. The system always has three production-free equilibria: 
\begin{equation}
    \mathbf{x_1} = (0, 0), \quad \mathbf{x_2}=(l,0), \quad \mathbf{x_3}=(1,0), 
\end{equation}
moreover, if $m \in (l,1)$ then there exists also the equilibrium  
\begin{equation}
    \mathbf{x_4}=(m, \varepsilon (m-l)(1-m)) \eqqcolon (\Bar{u}, \Bar{v}). 
\end{equation}
Notice that $\Bar{v} \in \mathcal{O}(\varepsilon)$. It is important to remark that the properties of the equilibria are only determined by the relationship between the parameters $l$ and $m$:  
\begin{itemize} 
    \item $\mathbf{x_1}$ is always locally asymptotically stable. In particular, its basin of attraction contains the set $\{u<l\}$; indeed, $u<l$ implies $\dot{u} \le 0$, $\dot{u}=0$ when $u=0$, and $v$ decreases exponentially towards 0 on the set $\{u=0\}$. 
    \item $\mathbf{x_2}$ is always unstable. However, if $m>l$, then it possesses a 1D stable manifold, locally described by the eigenspace $E(-\gamma (m-l))=\textnormal{span}_\mathbb{R}\{(l, \varepsilon l (1-l) + \gamma (m-l))\}$. 
    \item $\mathbf{x_3}$ is locally asymptotically stable if $m>1$, otherwise it is unstable. However, it always has a stable manifold: the subset of the $u$-axis $\{l < u \le 1, \; v=0\}$. 
    \item $\mathbf{x_4}$ is unstable if $m \in (l, (l+1)/2)$, while it locally asymptotically stable if $m \in ((l+1)/2, 1)$. In particular, since the eigenvalues of the Jacobian matrix of \eqref{eq2} computed at $\mathbf{x_4}$ are 
    \begin{equation} \label{eqlambda}
        \lambda_{1,2} = \frac{1}{2} \left(\varepsilon m (1+l-2m) \pm \sqrt{\varepsilon^2 m^2 (1+l-2m)^2 - 4 \varepsilon\gamma m (m-l)(1-m)}\right), 
    \end{equation}
    this equilibrium is a focus; note indeed that $0<\varepsilon \ll 1$ implies that the argument of the square root is negative. We will provide more details about this equilibrium in Section \ref{sec:interm_scale}. 
\end{itemize}
For the bifurcation analysis we focus on the role of $m$. For $m=l$, where $\Bar{v}=0$, there is a transcritical bifurcation, with $\mathbf{x_4}$ coinciding with $\mathbf{x_2}$. As $m$ increases and reaches a value $\Tilde{m}(l, \gamma,\varepsilon) \in (l, (l+1)/2)$ a stable heteroclinic cycle between $\mathbf{x_2}$ and $\mathbf{x_3}$ is formed; one orbit of such cycle coincides with the part of the $u$-axis $u\in [l,1]$. We will show in Section \ref{section3_4} that 
\begin{equation*}
    \tilde{m}(l,\gamma,\varepsilon) \xrightarrow{\varepsilon \to 0^+} \frac{l-1}{\log l},
\end{equation*}
where $\log$ indicates the natural logarithm. Such cycle shrinks around the unstable equilibrium $\mathbf{x_4}$ as $m$ increases and collapses on it for $m=(l+1)/2$, changing the stability of $\mathbf{x_4}$ through a Hopf bifurcation. The equilibrium $\mathbf{x_4}$ thus becomes locally stable and the cycle ceases to exist. Finally, for $m=1$, where again $\Bar{v}=0$, the system undergoes another transcritical bifurcation, with $\mathbf{x_4}$ collapsing on $\mathbf{x_3}$. 

For future use, let us introduce the curve 
\begin{equation} \label{future}
    v = \alpha(u) \coloneqq  \gamma \left(l - u + m \log\left(\frac{u}{l}\right)\right),
\end{equation}
\correz{which, if $m>l$, as we will demonstrate in Section \ref{section3}, describes the orbit converging to the equilibrium $\mathbf{x_2}$ in the reduced system obtained by setting $\varepsilon=0$.} Figure \ref{fig:sketch} shows the attractors of the orbits of system \eqref{eq2} as $m$ varies. It also contains several information that will be derived in Section \ref{section3}.

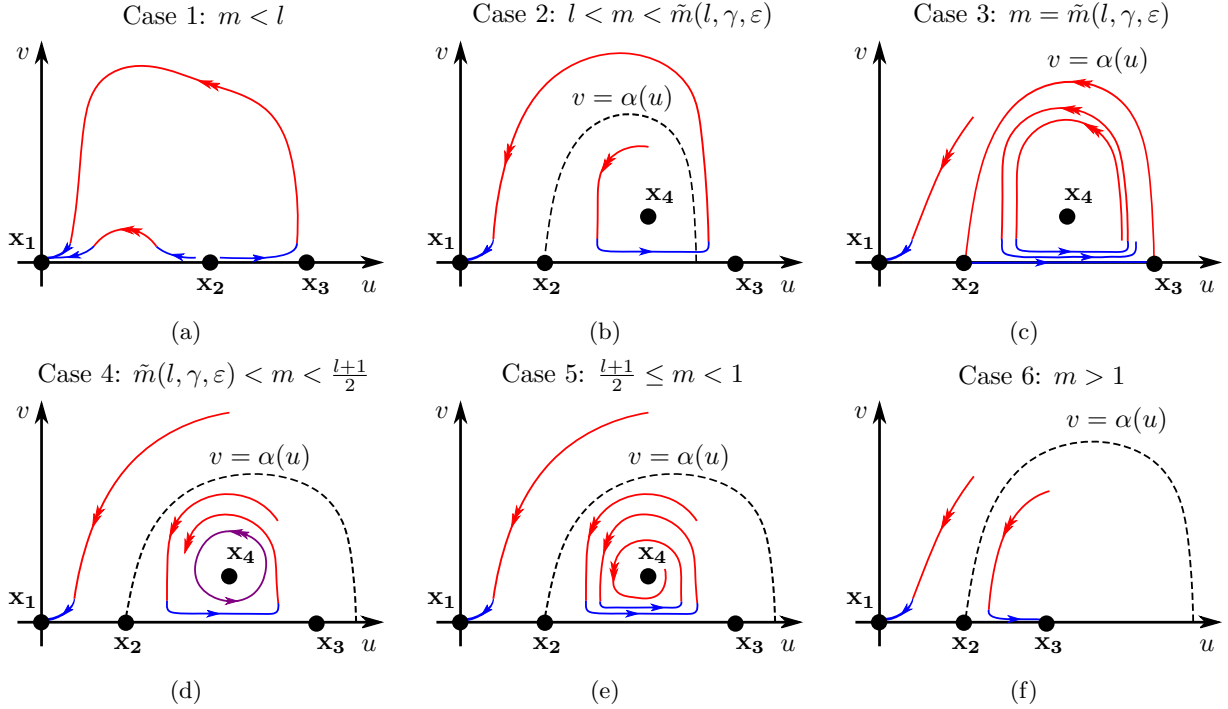
\begin{figure}[h!]
\centering
\begin{subfigure}{.3\textwidth}
  \centering
  \begin{tikzpicture}
 \node at (0,0) {\includegraphics[width=\linewidth]{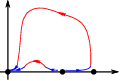}};
 \node at (-2.4,-1) {$\mathbf{x_1}$};
\node at (0.1,-1.6) {$\mathbf{x_2}$};
\node at (1.5,-1.6) {$\mathbf{x_3}$};
\node at (2.2,-1.6) {$u$};
\node at (-2.4,1.5) {$v$};
\node at (0,2) {Case 1: $m<l$};
  \end{tikzpicture}
  \caption{}
  \label{fig:sketch1}
\end{subfigure}\hspace{.5cm}
\begin{subfigure}{.3\textwidth}
  \centering
 \begin{tikzpicture}
 \node at (0,0) {\includegraphics[width=\linewidth]{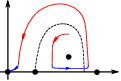}};
  \node at (-2.4,-1) {$\mathbf{x_1}$};
\node at (1.7,-1.6) {$\mathbf{x_3}$};
\node at (-1,-1.6) {$\mathbf{x_2}$};
\node at (0.5,-0.4) {$\mathbf{x_4}$};
\node at (2.2,-1.6) {$u$};
\node at (-2.4,1.5) {$v$};
\node at (0,0.9) {$v=\alpha(u)$};
\node at (0,2) {Case 2: $l<m<\tilde{m}(l,\gamma,\varepsilon)$};
  \end{tikzpicture}
  \caption{}
  \label{fig:sketch2}
\end{subfigure}\hspace{.5cm}
\begin{subfigure}{.3\textwidth}
  \centering
  \begin{tikzpicture}
 \node at (0,0) {\includegraphics[width=\linewidth]{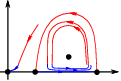}};
  \node at (-2.4,-1) {$\mathbf{x_1}$};
\node at (1.7,-1.6) {$\mathbf{x_3}$};
\node at (-1,-1.6) {$\mathbf{x_2}$};
\node at (0.5,-0.4) {$\mathbf{x_4}$};
\node at (2.2,-1.6) {$u$};
\node at (-2.4,1.5) {$v$};
\node at (0.75,1.4) {$v=\alpha(u)$};
\node at (0,2) {Case 3: $m=\tilde{m}(l,\gamma,\varepsilon)$};
  \end{tikzpicture}
  \caption{}
  \label{fig:sketch4}
\end{subfigure}\\
\begin{subfigure}{.3\textwidth}
  \centering
 \begin{tikzpicture}
 \node at (0,0) {\includegraphics[width=\linewidth]{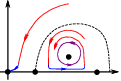}};
  \node at (-2.4,-1) {$\mathbf{x_1}$};
\node at (1.7,-1.6) {$\mathbf{x_3}$};
\node at (-1,-1.6) {$\mathbf{x_2}$};
\node at (0.5,-0.4) {$\mathbf{x_4}$};
\node at (2.2,-1.6) {$u$};
\node at (-2.4,1.5) {$v$};
\node at (0.75,0.9) {$v=\alpha(u)$};
\node at (0,2) {Case 4: $\tilde{m}(l,\gamma,\varepsilon)<m<\frac{l+1}{2}$};
  \end{tikzpicture}
  \caption{}
  \label{fig:sketch3}
\end{subfigure}\hspace{.5cm}
\begin{subfigure}{.3\textwidth}
  \centering
 \begin{tikzpicture}
 \node at (0,0) {\includegraphics[width=\linewidth]{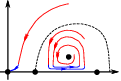}};
  \node at (-2.4,-1) {$\mathbf{x_1}$};
\node at (1.7,-1.6) {$\mathbf{x_3}$};
\node at (-1,-1.6) {$\mathbf{x_2}$};
\node at (0.4,-0.4) {$\mathbf{x_4}$};
\node at (2.2,-1.6) {$u$};
\node at (-2.4,1.5) {$v$};
\node at (0.75,0.9) {$v=\alpha(u)$};
\node at (0,2) {Case 5: $\frac{l+1}{2}\leq m<1$};
  \end{tikzpicture}
  \caption{}
  \label{fig:sketch5}
\end{subfigure}\hspace{.5cm}
\begin{subfigure}{.3\textwidth}
  \centering    \begin{tikzpicture}
 \node at (0,0) {\includegraphics[width=\linewidth]{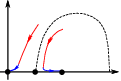}};
 \node at (-2.4,-1) {$\mathbf{x_1}$};
\node at (0.1,-1.6) {$\mathbf{x_3}$};
\node at (-1,-1.6) {$\mathbf{x_2}$};
\node at (2.2,-1.6) {$u$};
\node at (-2.4,1.5) {$v$};
\node at (1,1.4) {$v=\alpha(u)$};
\node at (0,2) {Case 6: $m>1$};
  \end{tikzpicture}
  \caption{}
  \label{fig:sketch6}
\end{subfigure}
\caption{Visualization of the dynamics as $m$ varies. Slow parts of the orbits are depicted in blue and with a single arrow along them, fast parts in red and with a double arrow. (a) $m<l$, the dynamic is attracted to $\mathbf{x_1}$, either immediately or after a slow permanence close to the $u$-axis, followed by a fast excursion away from it; (b) $l<m<\tilde{m}(l,\gamma,\varepsilon)$, same as (a), but the unstable equilibrium $\mathbf{x_4}$ appears in the economically relevant region \eqref{eq3}; (c) $m=\tilde{m}(l,\gamma.\varepsilon)$, the system exhibits a heteroclinic cycle connecting $\mathbf{x_2}$ and $\mathbf{x_3}$ in the slow flow and $\mathbf{x_3}$ and $\mathbf{x_2}$ in the fast flow (this latter heteroclinic orbit is approximated by the curve $v=\alpha(u)$, recall \eqref{future}); orbits starting inside this cycle are attracted to it, orbits starting outside are attracted to $\mathbf{x_1}$; (d) $\tilde{m}(l,\gamma,\varepsilon)<m<(l+1)/2$, orbits starting below the curve $v=\alpha(u)$ are attracted to the unique stable limit cycle (purple) around the unstable equilibrium $\mathbf{x_1}$, orbits starting above such curve are attracted to $\mathbf{x_1}$ (actually, we will show that there exists a value $\Bar{m}(l,\gamma,\varepsilon)$, such that $\Bar{m}(l,\gamma,\varepsilon) \to (l+1)/2$ as $\varepsilon \to 0$, which distinguishes between when the stable cycle enters the slow flow and when it does not, hence this case corresponds more precisely to $\tilde{m}<m<\Bar{m} \approx (l+1)/2$); (e) $(l+1)/2\leq m<1$, orbits starting below the curve $v=\alpha(u)$ are attracted to $\mathbf{x_4}$ (on which the limit cycle collapsed as $m\to {\frac{l+1}{2}}^-$), orbits starting above such curve are attracted to $\mathbf{x_1}$; (f) $m>1$, the dynamic is attracted either to $\mathbf{x_1}$ or to $\mathbf{x_3}$ (the curve $v=\alpha(u)$ again approximates the border between the two basins of attraction). \label{fig:sketch}}\end{figure}

\section{Multiple timescale analysis} \label{section3}

In this section, we analyze the two-timescale structure of system \eqref{eq2}. In order to highlight the two-timescale nature of the system, we write it as 
\begin{equation} \label{eq4}
    \dot{\mathbf{x}}=\mathbf{F_1}(\mathbf{x})+\varepsilon \; \mathbf{F_2}(\mathbf{x}), 
\end{equation}
where 
\begin{equation} \label{eq5}
    \mathbf{x}=
    \begin{bmatrix}
        u \\ v
    \end{bmatrix}, \quad
    \mathbf{F_1}(\mathbf{x})=
    \begin{bmatrix}
        -u v \\ 
        -\gamma v (m-u)
    \end{bmatrix}, \quad
    \mathbf{F_2}(\mathbf{x})=
    \begin{bmatrix}
        u (u-l) (1-u) \\ 0 
    \end{bmatrix}. 
\end{equation}

\subsection{Preliminaries on Geometric Singular Perturbation Theory} \label{section3_1} 

We provide a very brief introduction to GSPT (see e.g. \cite{40,23,wechselberger2020geometric,2} for a more in-depth presentation), and in particular of the entry-exit function \cite{35}. Both will be fundamental to study the combination of the two different timescales of system \eqref{eq4}.

Consider the so-called \emph{fast-slow system} in \emph{standard form}
\begin{align} \label{eq6}
\left\{
\begin{aligned}
    \varepsilon \; \mathbf{x}' &= \mathbf{f}(\mathbf{x}, \mathbf{y}, \varepsilon), \\ 
    \mathbf{y}' &= \mathbf{g}(\mathbf{x}, \mathbf{y}, \varepsilon),
\end{aligned}
\right.
\end{align}
where $\mathbf{x}\in \mathbb{R}^n$, $\mathbf{y}\in \mathbb{R}^m$, $\mathbf{f}\colon \mathbb{R}^{n+m+1} \to \mathbb{R}^n$, $\mathbf{g}\colon \mathbb{R}^{n+m+1} \to \mathbb{R}^m$, $\mathbf{f}, \mathbf{g} \in C^r$ for some $r$ sufficiently large, and $0<\varepsilon \ll 1$ is a small parameter. The variable $\mathbf{x}$ is called fast variable while the variable $\mathbf{y}$ is called slow variable. System \eqref{eq6} is formulated on the slow time $\tau$, and the $'$ indicates the derivative with respect to $\tau$. By defining the fast time $t=\tau/\varepsilon$, system \eqref{eq6} can be rewritten as
\begin{align} \label{eq7}
\left\{
\begin{aligned}
    \dot{\mathbf{x}} &= \mathbf{f}(\mathbf{x}, \mathbf{y}, \varepsilon), \\ 
    \dot{\mathbf{y}} &= \varepsilon \; \mathbf{g}(\mathbf{x}, \mathbf{y}, \varepsilon),
\end{aligned}
\right.
\end{align}
where the overdot $^\cdot$ now indicates the derivative with respect to $t$. The slow subsystem is defined by considering $\varepsilon=0$ in \eqref{eq6}, which yields 
\begin{align} \label{eq8}
\left\{
\begin{aligned}
    \mathbf{0} &= \mathbf{f}(\mathbf{x}, \mathbf{y}, 0), \\ 
    \mathbf{y}' &= \mathbf{g}(\mathbf{x}, \mathbf{y}, 0).
\end{aligned}
\right.
\end{align}
The slow flow defined by \eqref{eq8} is restricted to the critical manifold 
\begin{equation}
    \mathcal{C}_0 \coloneqq \{(\mathbf{x}, \mathbf{y})\in \mathbb{R}^{n+m} \colon \; \mathbf{f}(\mathbf{x}, \mathbf{y},0)=\mathbf{0}\}, 
\end{equation}
whose points are the equilibria of the fast subsystem 
\begin{align} 
\left\{
\begin{aligned}
    \dot{\mathbf{x}} &= \mathbf{f}(\mathbf{x}, \mathbf{y}, 0), \\ 
    \dot{\mathbf{y}} &= \mathbf{0}.
\end{aligned}
\right.
\end{align}
We provide now two definitions which will be fundamental for the analysis of our model \cite{40}. 

\begin{definition} \label{def1}
    A subset $\mathcal{M}_0 \subset \mathcal{C}_0$ is called \emph{normally hyperbolic} if the $n \times n$ matrix $\textnormal{D}_\mathbf{x}\mathbf{f}(\mathbf{x}, \mathbf{y}, 0)$ of first partial derivatives with respect to the fast variables has no eigenvalues with zero real part for all $(\mathbf{x}, \mathbf{y}) \in \mathcal{M}_0$.
\end{definition}

\begin{definition} \label{def2}
    A normally hyperbolic subset $\mathcal{M}_0 \subset \mathcal{C}_0$ is called \emph{attracting} if all eigenvalues of $\textnormal{D}_\mathbf{x}\mathbf{f}(\mathbf{x}, \mathbf{y}, 0)$ have negative real part for all $(\mathbf{x}, \mathbf{y}) \in \mathcal{M}_0$; similarly, $\mathcal{M}_0$ is called \emph{repelling} if all eigenvalues have positive real part. If $\mathcal{M}_0$ is normally
    hyperbolic and neither attracting nor repelling, it is of \emph{saddle type}.
\end{definition}

A basic result of GSPT is Fenichel's Theorem \cite[Theorem 3.1.4]{40} (see also \cite{23}). 

\begin{theorem}[Fenichel] \label{teo1}
    Consider a compact submanifold (possibly with boundary) $\mathcal{M}_0$ of the critical manifold $\mathcal{C}_0$. If $\mathcal{M}_0$ is normally hyperbolic, then for $\varepsilon>0$ sufficiently small, the following hold:
    \begin{enumerate}
        \item there exists a locally invariant manifold $\mathcal{M}_\varepsilon$, called slow manifold,  diffeomorphic to $\mathcal{M}_0$ (local invariance means that trajectories can enter or leave $\mathcal{M}_\varepsilon$ only through its boundaries); 
        \item $\mathcal{M}_\varepsilon$ is $\mathcal{O}(\varepsilon)$-close to $\mathcal{M}_0$; 
        \item the flow on $\mathcal{M}_\varepsilon$ converges to the slow flow as $\varepsilon \to 0$; 
        \item $\mathcal{M}_\varepsilon$ is $C^r$-smooth; 
        \item $\mathcal{M}_\varepsilon$ is normally hyperbolic and has the same stability properties with respect
        to the fast variables as $\mathcal{M}_0$ (attracting, repelling, or of saddle type); 
        \item $\mathcal{M}_\varepsilon$ is usually not unique but all the possible choices lie $\mathcal{O}(\exp({-D/\varepsilon}))$-close to
        each other, for some $D > 0$; 
        \item the stable and unstable manifolds of $\mathcal{M}_\varepsilon$ are locally invariant and are also $\mathcal{O}(\varepsilon)$-close and diffeomorphic to the stable and unstable manifolds of $\mathcal{M}_0$. 
    \end{enumerate}
\end{theorem}

Note that point 6 of Theorem \ref{teo1} implies that the specific choice of the slow manifold $\mathcal{M}_\varepsilon$ does not change analytical and numerical results.

As we mentioned above, fast–slow systems like \eqref{eq6} and \eqref{eq7} are said to be in standard form. In a more general context, it is possible
to analyze a fast–slow system in \emph{non-standard form} given by \cite{wechselberger2020geometric}
\begin{equation}\label{eq:nonstand}
    \dot{\mathbf{z}} = \mathbf{F}(\mathbf{z},\varepsilon), 
\end{equation}
with $\mathbf{z} \in \mathbb{R}^{n+m}$, $\mathbf{F}\colon \mathbb{R}^{n+m+1} \to \mathbb{R}^{n+m}$, and $\mathbf{F} \in C^r$, where the timescale separation is not explicit nor global. A system in the form \eqref{eq:nonstand} is singularly perturbed if the set
\begin{equation*}
    \mathcal{C}_0:=\{ \mathbf{z}\in  \mathbb{R}^{n+m} \colon \; \mathbf{F}(\mathbf{z},0)=\mathbf{0}\}
\end{equation*}
is non-empty, nor consists of isolated singularities. In particular, system \eqref{eq4} is in such non-standard form. However, sufficiently close to the critical manifold, we will be able to introduce a change of coordinates that brings our system in standard form, in order to apply the results presented in this section. \correz{Note that this means that the separation of variables into fast and slow is only a \emph{local} property close to the critical manifold $\mathcal{C}_0$, rather than a \emph{global} one. We will make this statement more precise in Sections \ref{section3_2} and \ref{section3_3}.}

\begin{figure}[h!]
    \centering 
     \begin{tikzpicture}
 \node at (0,0) {\includegraphics[width=0.3\linewidth]{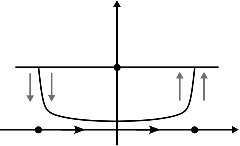}};
\node at (0,1.7) {$x$};
\node at (-1.5,0.3) {$x=x_0$};
\node at (-1.6,-1.5) {$y_0$};
\node at (1.5,-1.5) {$p_{\varepsilon}(y_0)$};
\node at (2.6,-1.2) {$y$};
  \end{tikzpicture}
    \caption{Visualization of the entry–exit map on the line $\{x=x_0\}$. \label{fig_3}}
\end{figure}
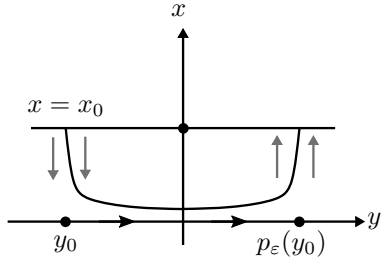

Consider now the planar system 
\begin{align} \label{eq9}
\left\{
\begin{aligned}
    \dot{x} &= x \; f(x,y,\varepsilon), \\ 
    \dot{y} &= \varepsilon \; g(x,y,\varepsilon),
\end{aligned}
\right.
\end{align}
with $(x,y) \in \mathbb{R}^2$, $g(0,y,0)>0$, and $\textnormal{sign}(f(0,y,0))=\textnormal{sign}(y)$. Notice that for $\varepsilon=0$ the $y$-axis consists of attracting/repelling equilibria if $y$ is negative/positive, respectively. Consider an orbit starting at $(x_0,y_0)$ for $x_0 \in \mathcal{O}(\varepsilon)$ and $y_0<0$ (see Figure \ref{fig_3}). Intuitively, we expect that it is attracted to the $y$-axis as long as $y<0$ and that it will be then repelled away when $y>0$. Note that, since $g(0,y,0)>0$, we expect the $y$-coordinate of the orbit to grow during this process. However, since the $y$-axis is not normally hyperbolic, Fenichel's Theorem \ref{teo1} can not explain this behavior since it could be applied only to a submanifold $\mathcal{M}_0$ of such axis away from the origin. On the other hand, the entry-exit function \cite{35} gives, in the form of a Poincaré map, an estimate of the behavior of such orbits near the origin. Consider the horizontal line $\{x=x_0\}$, the orbit of Figure \ref{fig_3} re-intersects that line for $y=p_0(y_0) + \mathcal{O}(\varepsilon)$, where $p_0(y_0)$ is defined implicitly as the non-trivial solution to \cite{34} 
\begin{equation} \label{eq10}
    \int_{y_0}^{p_0(y_0)} \frac{f(0,y,0)}{g(0,y,0)}dy=0.
\end{equation}
It is important to remark that \eqref{eq10} provides an approximation of the exit point, indeed it describes the orbits close to the $y$-axis through the flow on the $y$-axis itself (such description is only valid for $\varepsilon$ sufficiently small, see point 3 of Fenichel's Theorem \ref{teo1}). Let $dy=g(0,y,0) \; d\tau$, since the function $g$ describes the growth of the $y$-coordinate, one can transform \eqref{eq10} into an integral equation which provides the (slow) exit time $\tau_E$: 
\begin{equation} \label{eq11}
    \int_0^{\tau_E} f(0,y(\tau),0) \; d\tau = 0. 
\end{equation}
On the $y$-axis, the eigenvalues of the Jacobian of the fast subsystem of \eqref{eq9} are $\lambda_1=f(0,y,0)$ and $\lambda_2=0$. The former is associated to the fast variable $x$, while the latter to the slow variable $y$ (and it is equal to zero since we are considering the fast flow). Therefore, \eqref{eq11} is equivalent to 
\begin{equation} \label{eq12}
    \int_0^{\tau_E} \lambda_1|_{y(\tau)} \; d\tau = 0. 
\end{equation}
Notice that $\lambda_1(y)<0$ if $y<0$, while $\lambda_1(y)>0$ if $y>0$. Indeed, the eigenvalue $\lambda_1(y)$ describes the change of stability of the $y$-axis.  

\subsection{Fast formulation} \label{section3_2}

Setting $\varepsilon=0$ in \eqref{eq4}, we obtain the corresponding fast subsystem as
\begin{align} \label{eq13}
\left\{
\begin{aligned}
    \dot{u} &= -uv, \\ 
    \dot{v} &= -\gamma v (m-u). 
\end{aligned}
\right.
\end{align}
The critical manifold $\mathcal{C}_0$ of \eqref{eq4} is defined as the set of the equilibria of \eqref{eq13}, namely 
\begin{equation}
    \mathcal{C}_0 \coloneqq \{(u,v) \in \mathbb{R}^2 \colon \; v=0\}. 
\end{equation}
In this limit, one may notice a striking similarity of system \eqref{eq13} with the SI subsystem of a classic SIR model with normalized population \cite{jardon2021geometric}. Indeed, the resources $u$ are depleted by the production $v$ with the same underlying mechanism as the ``recruiting'' of susceptible individuals by infected, and thus infectious, ones. In this comparison, $m^{-1}$ plays the role of the Basic Reproduction Number $R_0$: assuming $m^{-1}<u_0\leq 1$, if $m<1$ (i.e., if $R_0>1$), we observe a peak in the production before resources are depleted, otherwise production immediately decreases and tends asymptotically to 0. Refer to Figure \ref{fig:sketch} for a visualization.

In this section, we study the behavior of the solutions of \eqref{eq13}, the fast subsystem of \eqref{eq4}. Denote with $(u_0,v_0)$ the initial conditions and define $u_\infty$ and $v_\infty$ as 
\begin{equation}
    u_\infty=\lim_{t \to +\infty} u(t) \quad \textnormal{and} \quad v_\infty=\lim_{t \to +\infty} v(t), 
\end{equation}
when these limit exist (under the flow of system \eqref{eq13}). 

\begin{proposition}
The trajectories of system \eqref{eq13} converge to $\mathcal{C}_0$ as $t \to \infty$. 
\end{proposition}
\begin{proof}
Notice that the solutions of the fast system \eqref{eq13} naturally evolve inside $\Delta$ given by \eqref{eq3}. Since $\dot{u} \le 0$, there exists $u_\infty \in [0, u_0]$. Moreover, $\dot{u}+\dot{v}/\gamma \le 0$, therefore there exist also $v_\infty \ge 0$. Integrating $\dot{u}+\dot{v}/\gamma$ we obtain 
    \begin{equation*}
    \begin{split}
        - \infty  < u_\infty + \frac{v_\infty}{\gamma} - u_0 - \frac{v_0}{\gamma} = \int_0^\infty \left(\dot{u}(t)+\frac{\dot{v}(t)}{\gamma}\right) \; dt  = -m \int_0^\infty v(t) \;  dt  < 0, 
    \end{split}
    \end{equation*}
    therefore $v_\infty=0$. 
\end{proof} 

The following proposition provides an implicit expression for $u_\infty$. \correz{The constant of motion $\Gamma(u,v)$ defined in \eqref{eq14} is a classic constant of motion for the SIR system, see e.g. \cite{jardon2021geometric}.}

\begin{proposition}
The quantity 
\begin{equation} \label{eq14}
    \Gamma(u,v)=m\log{u} - u - \frac{v}{\gamma} 
\end{equation}
is a constant of motion for system \eqref{eq13}. Consequently, $u_\infty \in (0, m)$. 
\end{proposition}
\begin{proof}
By direct derivation with respect to time, one can see that $\dot{\Gamma}(u,v) \equiv 0$, therefore 
\begin{equation} \label{eq14_bis}
    m\log{\frac{u_\infty}{u_0}} = u_\infty - u_0 - \frac{v_0}{\gamma}.  
\end{equation}
Define over $(0, u_0)$ the function $h(x)=m\log{(x/u_0)} - x + u_0 + v_0/\gamma$, notice that $\lim_{x \to 0^+} h(x) = -\infty$ while $h(u_0)>0$. Since $\frac{dh}{dx}(x)=m/x - 1$, $h$ increases in $(0, \min {\{m, u_0\}})$ and decreases in $(\min {\{m, u_0\}}, u_0)$, hence there exists a unique zero of $h$, which by definition coincides \correz{with} $u_\infty$, in the open interval $(0,m)$. 
\end{proof}

The fact that $u_\infty < m$ is not surprising. Indeed, the eigenvalue associated to the $v$ equation of \eqref{eq13} is $\lambda=-\gamma (m-u)$. Hence, the critical manifold is attracting for $u<m$ and repelling for $u>m$ (see Definition \ref{def2}). \correz{As we will show shortly in Section \ref{section3_3}, in an suitable neighborhood of $\mathcal{C}_0$, it is possible to introduce a change of coordinates that brings system \eqref{eq2} to standard GSPT form, where the rescaled $v$ is the fast variable and the rescaled $u$ the slow variable. This is in accordance with the fact that the local stability of $\mathcal{C}_0$ depends uniquely on the value of $u$, as we just showed.} 

Consider \eqref{eq14_bis}, and assume that $v_0\in \mathcal{O}(\varepsilon)$; we can thus ignore it, and derive the following relation between $u_\infty$ and $u_0$:
\begin{equation}\label{eq:riscrit}
m\log u_\infty - u_\infty=  m\log u_0 - u_0.
\end{equation}
\correz{Assuming} that $m<1$, for future use, let us define the map 
\begin{equation} \label{map1}
    \Pi_1 \colon \{u \in (m, 1]\}  \to \{u \in (0, m)\}    
\end{equation}
that maps $u_0$ into $u_\infty$ according to \eqref{eq:riscrit}. 

Until now, we have studied the flow of \eqref{eq13}, the fast subsystem of \eqref{eq4}. Before trying to understand the relationship between the orbits of the two systems, we will focus on the slow flow occurring near the critical manifold $\mathcal{C}_0$.

\subsection{Slow formulation} \label{section3_3}

Consider \eqref{eq2} and assume that a solution reached an $\mathcal{O}(\varepsilon^2)$-neighborhood of the critical manifold $\mathcal{C}_0$, namely $v \in \mathcal{O}(\varepsilon^2)$, for $u=u_\infty^\varepsilon$. In this situation, the influence of $\mathbf{F_2}$ \eqref{eq5} becomes very relevant. We rescale $v$ as  $v=\varepsilon x$ and apply a rescaling to the time variable, bringing the system to the slow timescale $\tau=\varepsilon t$: 
\begin{align} \label{eq15}
\left\{
\begin{aligned}
    u' &= u(u-l)(1-u)-ux, \\ 
    \varepsilon x' &= -\gamma x (m-u), 
\end{aligned}
\right.
\end{align}
where the \correz{$'$} indicates the derivative with respect to the slow time $\tau$. Notice that $v \in \mathcal{O}(\varepsilon^2)$ implies $x \in \mathcal{O}(\varepsilon)$. 

If we look at system \eqref{eq15} on the critical manifold $\mathcal{C}_0$, now determined by $x=0$, we obtain 
\begin{equation} \label{eq16}
    u' = u(u-l)(1-u).  
\end{equation}
Hence, on $\mathcal{C}_0$, $u$ decreases towards 0 if $0<u_\infty^\varepsilon < l$, while $u$ grows towards 1 if $u_\infty^\varepsilon>l$. Moreover, point 3 of Fenichel's Theorem \ref{teo1} implies that $u$ behaves like the solution to \eqref{eq16} on the slow timescale and as long as we are observing the orbits away from the non-hyperbolic point of the critical manifold $u=m$. Notice indeed that if $\varepsilon \to 0$ then $x \to 0$ and the evolution of $u$ \eqref{eq15} converges to the evolution described by Eq. \eqref{eq16}. This is also in line with point 2 of Fenichel's Theorem \ref{teo1} which tells us that the slow flow occurs if $x$ is $\mathcal{O}(\varepsilon)$-close to $\mathcal{C}_0$. 

Consider the particular case $l<m<1$: this is the only scenario where the entry-exit phenomenon described in Section \ref{section3_1} could happen. Assume that an orbit of the perturbed system \eqref{eq2} enters a $\mathcal{O}(\varepsilon^2)$-neighborhood of $\mathcal{C}_0$ for some $u_\infty^\varepsilon \in (l,m)$ (recall Eq. \eqref{eq14_bis}). \correz{Defining} $x=v/\varepsilon$ and $y=u-m$, system \eqref{eq2} can be written as
\begin{align} \label{eq18}
\left\{
\begin{aligned}
    \dot{y} &= \varepsilon \left( (y+m)(y+m-l)(1-y-m)-(y+m)x \right) \eqqcolon \varepsilon g(x,y,\varepsilon), \\ 
    \dot{x} &= x \gamma y \eqqcolon x f(x,y,\varepsilon), \\ 
\end{aligned}
\right.
\end{align}
which has the exact same structure of system \eqref{eq9}. In particular, $g(0, y, 0) > 0$ and $\textnormal{sign}(f (0, y, 0)) = \textnormal{sign}(y)$ for $y\in (l-m,1-m)$, i.e., $u\in (l,1)$. The entry-exit function can therefore be used to compute $\mathcal{O}(\varepsilon)$-approximations of the exit point $u_E^\varepsilon \in (m,1)$ and of the (slow) exit time $\tau_E$. Let $y_\infty^\varepsilon=u_\infty^\varepsilon-m$, then $y_E^\varepsilon=u_E^\varepsilon-m=p_0(y_\infty^\varepsilon)$ is the non-trivial solution to (see Eq. \eqref{eq10})
\begin{equation} \label{eq19}
    \int_{y_\infty^\varepsilon}^{p_0(y_\infty^\varepsilon)} \frac{\gamma y}{(y+m)(y+m-l)(1-y-m)} dy = 0,  
\end{equation}
which can be rewritten as 
\begin{equation} \label{eq20}
    l (1-m) \log \left( \frac{1-m-y_E^\varepsilon}{1-m-y_\infty^\varepsilon} \right) + m (l-1) \log \left( \frac{m+y_E^\varepsilon}{m+y_\infty^\varepsilon} \right) + (m-l) \log \left( \frac{y_E^\varepsilon+m-l}{y_\infty^\varepsilon+m-l} \right) = 0. 
\end{equation}
Applying the inverse change of variables $u=y+m$, we obtain
\begin{equation} \label{eq20bis}
    l (1-m) \log \left( \frac{1-u_E^\varepsilon}{1-u_\infty^\varepsilon} \right) + m (l-1) \log \left( \frac{u_E^\varepsilon}{u_\infty^\varepsilon} \right) + (m-l) \log \left( \frac{u_E^\varepsilon-l}{u_\infty^\varepsilon-l} \right) = 0. 
\end{equation} 
Intuitively, if $u_\infty^\varepsilon \to l^+$, necessarily $u_E^\varepsilon\to 1^-$, to compensate a term that diverges to $-\infty$; this happens regardless of the relation between the parameters $l<m$. This entry-exit relation, for an entry in the slow flow close to $l^+$, holds true whether the perturbed system \eqref{eq2} exhibits a stable limit cycle, a heteroclinic connection between $u=1$ and $u=l$, the stable equilibrium point $\mathbf{x_4}$, or neither (see Figures \ref{fig:sketch3}, \ref{fig:sketch4}, \ref{fig:sketch5}, and \ref{fig:sketch2}, respectively). In particular, this explains why and how, when $m\to \tilde{m}^+$, the stable limit cycle exhibited by system \eqref{eq2} ``collapses'' on the union of two heteroclinic orbits between $\mathbf{x_3}=(1,0)$ and $\mathbf{x_2}=(l,0)$ (in the fast flow, along the level \correz{curve} $\Gamma(u,v)\equiv -1$ of the constant of motion \eqref{eq14}) and between $\mathbf{x_2}=(l,0)$ and $\mathbf{x_3}=(1,0)$ (in the slow flow, evolving in time according to \eqref{eq16}). Moreover, Eq. \eqref{eq11} ensures that the (slow) exit time $\tau_E$ is the non-trivial solution to 
\begin{equation} \label{eq21}
    0 = \int_0^{\tau_E} f(0,y(\tau),0) \; d\tau = \int_0^{\tau_E} \gamma \; y(\tau) \; d\tau = \int_0^{\tau_E} -\gamma \left( m - u(\tau) \right) \; d\tau, 
\end{equation}
where the argument of the last integral is indeed equal to the eigenvalue associated to the fast variable $v$, which describes the stability of the critical manifold $\mathcal{C}_0$ (see Eq. \eqref{eq12} and Section \ref{section3_2}). Since on $\mathcal{C}_0$ the variable $u$ grows towards $1$ if $u_\infty>l$ (see Eq. \eqref{eq16}), \eqref{eq21} admits a unique non-trivial solution, hence the same is true also for \eqref{eq19}. Notice that, since $u_\infty^\varepsilon \to l^+$ implies $u_E^\varepsilon \to 1^-$, if $u_\infty^\varepsilon \to l^+$ then $\tau_E \to \infty$, indeed the orbit takes infinite time to travel the heteroclinic trajectory from $\mathbf{x_2}$ to $\mathbf{x_3}$. Again for future use (recall \eqref{map1}), for $l<m<1$ and in the limit $\varepsilon \to 0$ we define the map 
\begin{equation} \label{map2}
    \Pi_2 \colon \{ u \in (l, m)\} \to \{ u \in (m,1) \}
\end{equation}
that maps $u_\infty$ into $u_E$ according to \eqref{eq20bis}.

\correz{We remark that the change of coordinates $y=u-m$, that brings the non-hyperbolic point to the origin, is not strictly necessary, as shown e.g. in \cite{hsu2017bifurcation}; however, here we opted for bringing our system in the most common entry-exit setup, for ease of interpretation.}

In the next section we will focus on the interactions between the two different timescales in order to characterize the complete behavior of the orbits. 

\subsection{Unified formulation} \label{section3_4}

Standard perturbation theory \cite[Corollary 3.1.7]{2} implies that an orbit of the perturbed system \eqref{eq2}, away from the critical manifold $\mathcal{C}_0$, follows $\mathcal{O}(\varepsilon)$-closely the orbit of the fast system \eqref{eq13}, related to the same initial conditions, for $\mathcal{O}(1)$ times $t$. Even if under the flow of \eqref{eq13} $v$ converges to zero, it is not obvious that in the perturbed flow $v$ enters a $\mathcal{O}(\varepsilon^2)$-neighborhood of $\mathcal{C}_0$. Indeed, $v$ decreases at most exponentially, therefore it requires a time $t$ of order at least $|\log \varepsilon | \gg 1$ to enter such neighborhood. Hence, in order to understand the behavior of the perturbed orbits for larger times, several situations must be distinguished, according to the bifurcation analysis of our system (see Section \ref{section2_2}). We exploit techniques from Geometric Singular Perturbation Theory to describe in details the evolution of any orbit towards a stable equilibrium or limit cycle. 

\subsubsection{$m<l$} \label{section3_4_1}

The Poincaré–Bendixson Theorem implies that any orbit starting from the point $(u_0, v_0)$, $v_0>0$ converges towards the equilibrium $\mathbf{x_1}$. Hence, the orbits necessarily enter a $\mathcal{O}(\varepsilon^2)$-neighborhood of $\mathcal{C}_0$, starting the slow flow. Anyway, several situations must be distinguished, depending on the initial conditions (recall point 3 of Fenichel's Theorem \ref{teo1} and Eq. \eqref{eq16}). All these possible scenarios are represented in Figure \ref{fig:sketch1}. 

Assume that $(u_0, v_0)$ is $\mathcal{O}(1)$-away from $\mathcal{C}_0$; the orbit will follow the fast subsystem until it enters the slow flow for  some $u<m<l$ (recall that the critical manifold is attracting only for $u<m$). At this moment, $u$ decreases towards 0, hence the orbit converges to $\mathbf{x_1}$. 

Assume that $v_0 \in \mathcal{O}(\varepsilon^2)$, i.e., the orbits begins inside the slow flow. If $u_0<m$ then the orbits converges to $\mathbf{x_1}$ immediately. If $u_0 \in (m,l)$, then, as $u$ decreases, the orbit is repelled away from $\mathcal{C}_0$ as long as $u>m$, while $v$ starts to decrease when $u=m$. Hence ,the orbit could initially escape from the slow flow, but surely later it re-enters it. If $u_0>l$, then in the slow flow $u$ increases towards 1 and, at the same time, the orbit is repelled away from $\mathcal{C}_0$, hence it enters the fast flow. We then fall back in the first scenario. 

\subsubsection{$l<m<1$} \label{section3_4_3} 

This is the most complex and interesting case. It is indeed the only scenario where the entry-exit phenomenon described in Section \ref{section3_1} could happen (see also Section \ref{section3_3}). Moreover, the orbits could not enter the slow flow since they could be attracted to $\mathbf{x_4}$, which is not $\mathcal{O}(\varepsilon^2)$-close to $\mathcal{C}_0$, or by a stable limit cycle, which could also not be $\mathcal{O}(\varepsilon^2)$-close to $\mathcal{C}_0$ since it collapses on $\mathbf{x_4}$ as $m \to {\frac{l+1}{2}}^-$. Consider an orbit starting from the point $(u_0, v_0)$ $\mathcal{O}(1)$-away from the critical manifold. First, we want to understand when it converges to $\mathbf{x_1}$ and when to another attractor of the system. In order to distinguish between these two different possibilities, we rely on the quantity $\Gamma$, recall \eqref{eq14} and \eqref{eq14_bis}, as proven in the following proposition.  

\begin{proposition} \label{prop_4}
    Consider an orbit of system \eqref{eq13} starting from $(u_0, v_0)$ $\mathcal{O}(1)$-away from the critical manifold. Assuming that $m>l$, then $u_\infty \le l$ if and only if $v_0 \ge \alpha(u_0)$ \eqref{future}. In particular, if $m<(l-1)/\log l<1$, then 
    \begin{align} \label{eq22} 
        \begin{cases}
            \alpha(u_0) < 0 \quad & \textnormal{if} \; u_0 < l \; \textnormal{or} \; u_0>\Tilde{u}, \\ 
            \alpha(u_0) = 0 & \textnormal{if} \; u_0=l  \; \textnormal{or} \; u_0=\Tilde{u}, \\ 
            \alpha(u_0) > 0 & \textnormal{if} \; l < u_0 < \Tilde{u}, \\ 
        \end{cases} 
    \end{align}
    where $\Tilde{u}$ is the unique zero of $\alpha(u)$ in the interval $(m,1)$, otherwise if $m>(l-1)/\log l$ then 
    \begin{align} \label{eq17_bis} 
        \begin{cases}
            \alpha(u_0) < 0 \quad & \textnormal{if} \; u_0 < l, \\ 
            \alpha(u_0) = 0 & \textnormal{if} \; u_0=l, \\ 
            \alpha(u_0) > 0 & \textnormal{if} \; u_0>l. \\ 
        \end{cases} 
    \end{align}
\end{proposition}
\begin{proof}
    The curve $v_0 = \alpha(u_0)$ is obtained by setting $u_\infty = l$ in \eqref{eq14_bis}. Note that $\alpha(l)=0$. Since
    \begin{equation*}
        \frac{d \alpha(u)}{d u} = - \gamma + \gamma \frac{m}{u}, 
    \end{equation*}
    $\alpha$ is an increasing function of $u$ in the interval $(0,m)$, while it is decreasing if $u>m$. If $m>1$ the thesis is obvious, in particular $\alpha$ is always a strictly increasing function of $u$. Assuming that $m<1$, it suffices to show that there exists a zero of $\alpha$ in the interval $(m,1)$ if and only if $m<(l-1)/\log l$. This follows directly from the fact that $\alpha(1)=\gamma(l-1-m\log l)$, which is negative if and only if $m<(l-1)/\log l$. Finally, notice that $l<(l-1)/\log l<1$, which is in line with the assumption $m \in (l,1)$. 
\end{proof}

Note that the curve $v=\alpha(u)$ \eqref{future} provides an approximation of curve that separates the basins of attraction of $\mathbf{x_1}$ with the rest of the set $\Delta$ \eqref{eq3}, indeed the perturbed system is influenced also by $\varepsilon \mathbf{F_2}$ \eqref{eq4}. From now on, we assume that $\varepsilon$ is sufficiently small to justify the approximations used throughout the remainder of this section. If $v_0 > \alpha(u_0)$ the orbit should enter the slow flow for $u=u_\infty^\varepsilon<l$ and hence converge towards $\mathbf{x_1}$ ($u_\infty^\varepsilon$ is slightly different than $u_\infty$ \eqref{map1}, in Remark \ref{rmk1} we will provided more details about this approximation, which obviously becomes better as $\varepsilon$ decreases). On the other hand, if $v_0 < \alpha(u_0)$, several cases must be distinguished, depending on the value of $m$. As we mentioned, the slow flow does not necessarily begin. 
\begin{enumerate}
\item If $m=(l-1)/\log l$, then, in the limit $\varepsilon \to 0$, an heteroclinic orbit from $\mathbf{x_3}$ to $\mathbf{x_2}$ exists, and it is described by $v=\alpha(u)$. Indeed, in this case $\alpha(1)=0$ and, by definition, $\alpha(l)=0$. This means that $(l-1)/\log l$ represents an approximation, valid for $\varepsilon$ sufficiently small, of the value $\tilde{m}$ for which an heteroclinic cycle if formed. In this situation, any orbit starting from the point $(u_0, v_0)$, with $v_0 < \alpha(u_0)$, is attracted to such cycle. This situation is represented in Figure \ref{fig:sketch4}. Interestingly, $\tilde{m}$ loses its dependency on $\gamma$ as $\varepsilon \to 0$. 
\item If $l<m<(l-1)/\log l$, then the only attractor is $\mathbf{x_1}$, indeed $\mathbf{x_4}$ exists but is unstable and the cycle around it does not exist. Therefore any orbit starting from the point $(u_0, v_0)$, with $v_0 < \alpha(u_0)$, sooner or later will escape from the set $\{0<v<\alpha(u)\}$. As we will explain in details later, this happens when $u_E=\Pi_2(u_\infty)>\Tilde{u}$ \eqref{map2}. Indeed, we will prove that, in this case, the slow flow necessarily begins (recall that \eqref{eq22} holds). To be precise, this happens when $u_E^\varepsilon=\Pi_2(u_\infty^\varepsilon)>\Tilde{u}^\varepsilon$, where $\tilde{u}^\varepsilon$ represents the zero in the interval $(m,1)$ of the curve that separates the basin of attraction of $\mathbf{x_1}$ with the rest of the set $\Delta$ \eqref{eq3}. However, as mentioned above, we will assume that $\varepsilon$ is sufficiently small to be able to make such approximations (indeed, the values $u_\infty^\varepsilon$, $u_E^\varepsilon$, and $\tilde{u}^\varepsilon$ are unknown). This situation is represented in Figure \ref{fig:sketch2}. 
\item If $(l-1)/\log l < m < (l+1)/2$, then any orbit starting from the point $(u_0, v_0)$, with $v_0 < \alpha(u_0)$, is attracted to the stable limit cycle around the unstable equilibrium point $\mathbf{x_4}$. Note that \eqref{eq17_bis} holds, in contraposition to what happens in the case $l<m<(l-1)/\log l$. This situation is represented in Figure \ref{fig:sketch3}. 
\item If $(l+1)/2 \le m < 1$, then any orbit starting from the point $(u_0, v_0)$, with $v_0 < \alpha(u_0)$, is attracted to the stable equilibrium point $\mathbf{x_4}$. Again, recall that \eqref{eq17_bis} holds. This situation is represented in Figure \ref{fig:sketch5}. 
\end{enumerate}

Until now, given the initial point $(u_0, v_0)$, we have determined where an orbit converges \emph{asymptotically}. Now, we want to focus on the \emph{transient} behavior that it exhibits during this process. We will assume that $v_0 < \alpha(u_0)$ since the other case has already been fully described. Define 
\begin{equation} \label{isocline}
    \beta(u)\coloneqq\varepsilon (u-l)(1-u), 
\end{equation}
the $u$-isocline. Notice that, since $\varepsilon \ll 1$, then $\beta(u) < \alpha(u)$ for all $u \in (l,m]$. The following technical lemma will be useful to understand when the \correz{orbit enters the slow flow}. 

\begin{lemma} \label{lemma1}
    Consider \eqref{eq2} and assume that $l <m <1$. Consider an initial condition $(u_*, v_*)$ such that $l<u_*<m-K_1$, $0<K_1 \in \mathcal{O}(1)$, and $v_* < \beta(u_*)$ (in particular, $v_* < \alpha(u_*)$). Let $0<K_2<K_1$ such that $K_2, K_1-K_2 \in \mathcal{O}(1)$ and denote with $(u^*, v^*)$ the point where the corresponding trajectory intersects the line $\{u=m-K_2\}$. Then, for sufficiently small $\varepsilon$, we have that $v^* \in \mathcal{O}(\varepsilon^2)$. 
\end{lemma}
\begin{proof}
    As long as $u \in [u_*, m-K_2)$, $\dot{v} \le -\gamma K_2 v$. Hence, if $v$ becomes $\mathcal{O}(\varepsilon^2)$-small in a time shorter than the one needed by $u$ to attain the value $u^*=m-K_2$, the thesis follows. Indeed, during this process $v$ decreases because $u<m$, while $u$ increases because $v<\beta(u)$. Since $\dot{u} \le \varepsilon u$, Gr\"onwall's Lemma implies that $u$ takes a time $t$ of order at least equal to $1/\varepsilon$ to reach the value $u^*$. A function behaving like $v$, meaning a function that decreases exponentially, needs a time $t$ or order $|\log \varepsilon|$ to become $\mathcal{O}(\varepsilon^2)$-small. But $|\log \varepsilon| \ll 1/\varepsilon$, hence this happens before $u$ reaches the value $u^*$. 
\end{proof}

\begin{remark} \label{rmk1}
    Consider an orbit that exits from the slow flow for $u=u_E^\varepsilon$ and $v=v_E^\varepsilon < \alpha(u_E^\varepsilon)$. The value $u_\infty$ given by $\Pi_1$ \eqref{map1} (applied to $u_0=u_E^\varepsilon$) provides (at worst) an $\mathcal{O}(\varepsilon |\log \varepsilon|)$-approximation of both $u_*<m$, such that $v_*=\beta(u_*)$, and the actual value of $u$ for which the orbit re-enters the slow flow $u_\infty^\varepsilon$ (assuming that the hypothesis of Lemma \ref{lemma1} are satisfied, otherwise the slow flow might not begin). Indeed, firstly standard perturbation theory implies that the orbit of the fast subsystem \eqref{eq13} is $\mathcal{O}(\varepsilon)$-close to the one of the perturbed system \eqref{eq2} when it reaches the point $(u_*,v_*)$ since this happens in the fast flow, i.e., in $\mathcal{O}(1)$-times $t$. Moreover, under the flow of \eqref{eq13}, $u$ will not suffer large variations anymore since $u_*<m$ and $v_* \in \mathcal{O}(\varepsilon)$. Secondly, as observed in the proof of the lemma, an orbit takes a time $t$ of order $|\log \varepsilon|$ to travel from $(u_*, v_*)$ to an $\mathcal{O}(\varepsilon^2)$-neighborhood of $\mathcal{C}_0$. However, $u_* \le u(t) \le u_* \exp(\varepsilon t) = u_* + \mathcal{O}(\varepsilon |\log \varepsilon|)$. 
\end{remark} 

Consider an orbit of \eqref{eq2} starting from $(u_0,v_0)$ below the curve $v=\alpha(u)$ and not $\mathcal{O}(\varepsilon^2)$-close to $\mathcal{C}_0$. It will reach the isocline curve $v=\beta(u)$ \eqref{isocline} for some $u_*<m$ since it initially follows the orbit of \eqref{eq13}. Several situations must be distinguished depending on the value of $m$ (for clarity, they are presented in an order that progresses from the simplest to the most complex, and the enumeration matches the one previously discussed). 
\begin{enumerate}
\setcounter{enumi}{1}
\item If $l<m<\tilde{m}(l,\gamma,\varepsilon)\approx(l-1)/\log l$ and $u_*$ is $\mathcal{O}(1)$-away from $m$, then Lemma \ref{lemma1} and Remark \ref{rmk1} imply that the orbit enters the slow flow for some $u=u_\infty^\varepsilon \approx u_*<m$. Denote with $u_E^\varepsilon$ the value attained by $u$ when the orbit exits from the slow flow. If $u_E^\varepsilon \approx u_E=\Pi_2(u_\infty)>\Tilde{u}$, this means that the orbit exits from the slow flow above the curve $v=\alpha(u)$, hence it will directly converge towards $\mathbf{x_1}$. If instead $u_E^\varepsilon \approx u_E=\Pi_2(u_\infty^\varepsilon)<\Tilde{u}$, then this process will happen again. Indeed, the under the fast flow the orbit will reach again the curve $v=\beta(u)$ for some $u_*<m$ (see Remark \ref{rmk1}). In particular, since, sooner of later, the orbit must converge towards $\mathbf{x_1}$, we must have $\Pi_1(u_E)=\Pi_1(\Pi_2(u_\infty))<u_\infty$. This means that this process repeats $n$ times, after which we will have 
\begin{equation*}
    (\Pi_2 \circ \overbrace{\Pi_1 \circ \Pi_2 \circ \dots \circ \Pi_1 \circ \Pi_2}^{n-\textnormal{times}}) (u_\infty) > \tilde{u},
\end{equation*}
and the orbit starts to converge towards $\mathbf{x_1}$. The composition of maps $\Pi_2 \circ \Pi_1$ is therefore increasing while the composition of maps $\Pi_1 \circ \Pi_2$ is decreasing. This situation is represented in Figure \ref{fig:sketch2}. Obviously, if initially $u_*$ is $\mathcal{O}(\varepsilon)$-close to $m$, then the orbit will be repelled away from $\mathbf{x_4}=(\Bar{u}, \Bar{v})$ since it is unstable (notice that $\beta(m)=\beta(\Bar{u})=\Bar{v}$) and, sooner or later, we will fall back in the previous scenario. 
\setcounter{enumi}{0}
\item If $m=\tilde{m}(l,\gamma,\varepsilon)\approx(l-1)/\log l$, then the entry-exit phenomenon happens an infinite number of times. Since the orbit is attracted to the heteroclinic cycle, the map $\Pi_2 \circ \Pi_1$ is increasing while the map $\Pi_1 \circ \Pi_2$ is decreasing. This situation is represented in Figure \ref{fig:sketch4}. 
\setcounter{enumi}{3}
\item If $(l+1)/2 \le m < 1$ and $u_*$ is $\mathcal{O}(1)$-away from $m$ the entry-exit phenomenon happens $n$ times, after which we will have 
\begin{equation}
    (\overbrace{\Pi_1 \circ \Pi_2 \circ \dots \circ \Pi_1 \circ \Pi_2}^{n-\textnormal{times}}) (u_\infty) - m \in \mathcal{O}(\varepsilon). 
\end{equation}
At this point, since $\mathbf{x_4}$ is locally asymptotically stable, the orbit will be attracted to it (to be precise, this happens when the perturbed orbit reaches the $u$-isocline $\mathcal{O}(\varepsilon)$-close to $m$). In this situation, the map $\Pi_2 \circ \Pi_1$ is decreasing while the map $\Pi_1 \circ \Pi_2$ is increasing. This situation is represented in Figure \ref{fig:sketch5}. Obviously, if $u_*$ is $\mathcal{O}(\varepsilon)$-close to $m$ then the orbit will be immediately attracted to it. 
\setcounter{enumi}{2}
\item If $\tilde{m}(l,\gamma,\varepsilon)\approx(l-1)/\log l < m < (l+1)/2$ two situations must be distinguished. Indeed, since as $m \to {\frac{l+1}{2}}^-$ the stable limit cycle collapses on $\mathbf{x_4}$ while as $m \to {\frac{l-1}{\log l}}^+$ it stretches towards the heteroclinic cycle, there exists a value $\Bar{m}(l,\gamma,\varepsilon) \in ((l-1)/\log l, (l+1)/2)$ such that the cycle enters a $\mathcal{O}(\varepsilon^2)$-neighborhood of $\mathcal{C}_0$ if and only if $m<\Bar{m}$. Assume that $m>\Bar{m}$, and consider an orbit starting from the point $(u_*, v_*)$, $v_*=\beta(u_*)$ inside the stable limit cycle. Since it can not enter a $\mathcal{O}(\varepsilon^2)$-neighborhood of $\mathcal{C}_0$, necessarily $u_*$ is $\mathcal{O}(\varepsilon)$-close to $m$, otherwise Lemma \ref{lemma1} would imply that the orbit enters such neighborhood. This implies that if $m>\Bar{m}$, i.e., the limit cycle does not enter the slow flow, then the minimum and the maximum values attained by $u$ on such cycle are $\mathcal{O}(\varepsilon)$-close to $m$. Since the cycle is stable and the considered orbit is $\mathcal{O}(\varepsilon)$-close to it, the orbit will be attracted to the cycle. On the other hand, if such orbit starts outside the stable limit cycle, if $u_*$ is $\mathcal{O}(1)$-far from $m$, the the orbit will enter the slow flow and a process similar to the one the case $(l+1)/2 \le m < 1$ happens. The map $\Pi_2 \circ \Pi_1$ is therefore decreasing while the map $\Pi_1 \circ \Pi_2$ is increasing. Once $u_*$ is $\mathcal{O}(\varepsilon)$-close to $m$ it is also $\mathcal{O}(\varepsilon)$-close to the stable limit cycle and it will be attracted to it, hence the orbit does not enter the slow flow anymore. Assume now that $m<\Bar{m}$ (see Figure \ref{fig:sketch3}). Since the stable limit cycle enters the slow flow, the entry-exit process happens an infinite number of times. If an orbit starts outside such cycle the map $\Pi_2 \circ \Pi_1$ is decreasing while the map $\Pi_1 \circ \Pi_2$ is increasing. On the other hand, if an orbit starts inside such cycle the map $\Pi_2 \circ \Pi_1$ is increasing while the map $\Pi_1 \circ \Pi_2$ is decreasing (note that, in this case, if $u_*$ is too close to $m$ then the orbit will be initially repelled away from $\mathbf{x_4}$, recall that it is a focus, and the entry-exit phenomenon could not immediately start). 
\end{enumerate}

\subsubsection{$m>1$} \label{section3_4_2} 

In this case, the orbits can converge towards $\mathbf{x_1}$ or $\mathbf{x_3}$. In both cases, they enter a $\mathcal{O}(\varepsilon^2)$-neighborhood of $\mathcal{C}_0$, starting the slow flow. Since the critical manifold is attracting everywhere, if an orbit enters such neighborhood for $u<l$ then it will converge towards $\mathbf{x_1}$, otherwise towards $\mathbf{x_3}$ (see Eq. \eqref{eq16}). Again, the curve $v=\alpha(u)$ distinguishes between these two possibilities (see Proposition \ref{prop_4}). These two possibilities are represented in Figure \ref{fig:sketch6}.  

\subsubsection{Concluding remarks on the limit $\varepsilon \to 0$} \label{section3_4_4}

We conclude the section with a discussion regarding the reasoning behind the limit $\varepsilon \to 0$ that we took several times. Away from the critical manifold $\mathcal{C}_0$, taking the limit $\varepsilon \to 0$ means neglecting the (small) influence of the term $\varepsilon u (u-l)(1-u)$ when computing the (fast) time derivative of $u$, hence the flow can be approximated by system \eqref{eq13} and the map $\Pi_1$ \eqref{map1} can be derived if $m<1$. Close to $\mathcal{C}_0$ this approximation can not be made, since $v \in \mathcal{O}(\varepsilon^2)$. In this situation, the original system can be written as \eqref{eq15}. Taking the limit $\varepsilon \to 0$ means neglecting the influence of $v$ when computing the (slow) time derivative of $u$, hence the flow can be approximated by Eq. \eqref{eq16}. Indeed, during the slow flow, if $\varepsilon \to 0$ then $v \to 0$. 
Moreover, in the particular case $l<m<1$ the map $\Pi_2$ \eqref{map2} can be defined. In conclusion, approximating the fast flow with \eqref{eq13} and the slow flow with \eqref{eq16}, we obtain an approximation of the whole flow of \eqref{eq2}, which is exact in the limit $\varepsilon \to 0$.

\subsection{A singular Hopf bifurcation} \label{sec:interm_scale}

We showed that system \eqref{eq2} evolves on two different timescales, depending on the position of the orbits with respect to the critical manifold $\mathcal{C}_0$. However, the Hopf bifurcation described in Section \ref{section2_2} creates a sort of third ``intermediate'' timescale. Indeed, Hopf's Theorem states that the initial (i.e., for $m \to {\frac{l+1}{2}}^-$) amplitude of the stable limit cycle around $\mathbf{x_4}$ is zero but its initial period $T_t$ (in the fast timescale $t$) is equal to $2 \pi / |\textnormal{Im}(\lambda_{1,2})|$, where $\lambda_{1,2}$ are given by \eqref{eqlambda}. Hence, $T_t \in \mathcal{O}(1/\sqrt{\varepsilon})$, meaning that it is larger than the $\mathcal{O}(1)$-times $t$ describing the fast flow, while it is smaller than the $\mathcal{O}(1/\varepsilon)$-times $t$ describing the slow flow. The stable cycle can evolve on this intermediate timescale only if $\Bar{m} < m < (l+1)/2$, indeed if $(l-1)/\log l < m<\Bar{m}$ then the period of the cycle is of the order $\mathcal{O}(1/\varepsilon)$ since it enters the slow flow. Precisely, a singular Hopf bifurcation occurs, as described in \cite[Section 8.2]{40}, since 
\begin{align} 
\left\{
\begin{aligned}
    \lim_{\varepsilon \to 0} \textnormal{Im}(\lambda_{1,2}) &= 0 \quad \textnormal{on the fast timescale } t, \\ 
    \lim_{\varepsilon \to 0} |\textnormal{Im}(\lambda_{1,2})| &= \infty \quad \textnormal{on the slow timescale } \tau, 
\end{aligned}
\right.
\end{align}
i.e., the eigenvalues $\lambda_{1,2}$ become singular as $\varepsilon \to 0$. Indeed, on the slow timescale $\tau$, the initial period $T_\tau$ is of the order $\mathcal{O}(\sqrt{\varepsilon})$.

\section{Numerical simulations} \label{section4}

The underlying assumption behind the reasoning made in Section \ref{section3} is that $\varepsilon$ should be  ``sufficiently small''. In particular, the maps $\Pi_1$ \eqref{map1} and $\Pi_2$ \eqref{map2} are defined in the limit $\varepsilon \to 0$. In this section, through numerical simulations, we will give various visualizations of our analysis. 

Firstly, we will verify how well the curve $v=\alpha(u)$ \eqref{future} separates the basin of attraction of $\mathbf{x_1}$ with the rest of the set $\Delta$ \eqref{eq3} for different values of $\varepsilon$. 

Secondly, for a fixed value of $\varepsilon$, we will carry out a numerical bifurcation analysis of our model. In particular, we will analyze how the cycle expands as $m$ decreases, leading to the formation of an heteroclinic cycle between $\mathbf{x_2}$ and $\mathbf{x_3}$. 

Lastly, assuming that the stable limit cycle exists and that it enters the slow flow (i.e., $(l-1)/ \log l < m < \Bar{m}$), we will find the fixed point of the map $\Pi_1 \circ \Pi_2$. This value represents the value of $u$, in the limit $\varepsilon \to 0$, for which the stable limit cycle enters the slow flow. Note that we still do not know the value of $\Bar{m}$, which therefore has to be derived numerically as well. Moreover, we will verify how well the fixed point $u_\infty$ of $\Pi_1 \circ \Pi_2$ approximates the value $u_\infty^\varepsilon$ for which the stable limit cycle enters the slow flow for different values of $m$ and $\varepsilon$. Note that, looking at a numerical simulation, it is not possible to exactly tell when it enters the slow flow, indeed we only know that we must have $v \in \mathcal{O}(\varepsilon^2)$. However, as explained in Remark \ref{rmk1}, $u_\infty^\varepsilon$ can be approximated with $u_*$, which we will be able to measure. 

Since it is possible that $v$ becomes exponentially small (i.e., of the order of $\exp({-K/\varepsilon})$, for some $0<K\in \mathcal{O}(1)$) during the slow flow, it could be impossible to precisely numerically simulate system \eqref{eq2}. For this reason, we introduce the variable $w=\log v$ (see e.g. \cite[Appendix A.5]{della2024geometric}) in order to re-write that system as 
\begin{align} \label{eq23}
\left\{
\begin{aligned}
    \dot{u} &= \varepsilon u (u-l) (1-u)  - u \exp(w), \\ 
    \dot{w} &= -\gamma (m-u). \\ 
\end{aligned}
\right.
\end{align}
System \eqref{eq23} is remarkably less stiff than system \eqref{eq2}, as we shall see in the following.

\subsection{Approximating the basins of attraction: the curve $v=\alpha(u)$} \label{section4_1}

In order to understand how well the curve $v=\alpha(u)$ \eqref{future} separates the basin of attraction of $\mathbf{x_1}$ with the rest of the set $\Delta$ \eqref{eq3}, we fixed the values of the parameters $l=0.4$, $m=0.67$, $\gamma=1$ and we simulated the evolution of the orbits for hundreds of different initial conditions and for two different values of $\varepsilon$ (namely, $\varepsilon=0.02$ and $\varepsilon=0.01$). Figure \ref{fig:test_alpha} shows the results of this test. For such values of the parameters, there exists a stable cycle between the $u$-axis and the curve $v=\alpha(u)$, and the only other possible attractor is $\mathbf{x_1}$. Although the cycle attracts few orbits starting above the curve $v=\alpha(u)$, notice that this curve approximates well the border between the basin of attraction of $\mathbf{x_1}$ and the basin of attraction of the stable cycle. In particular, such orbits start from initial conditions which are not $\mathcal{O}(\varepsilon)$-far from $v=\alpha(u)$, indeed they are closer to this curve. Finally, notice that the approximation sharpens as $\varepsilon$ decreases. 

\begin{figure}[h!]
    \centering
\begin{subfigure}{.45\textwidth}
  \centering   
\begin{tikzpicture}
 \node at (0,0) {\includegraphics[width=\linewidth]{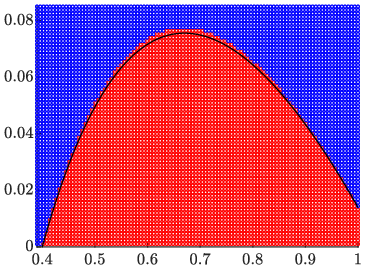}};
\node at (4,-2.5) {$u$};
\node at (-3.2,3) {$v$};
  \end{tikzpicture}
  \caption{}
\end{subfigure}\hspace{0.5cm}
\begin{subfigure}{.45\textwidth}
  \centering
\begin{tikzpicture}
 \node at (0,0) {\includegraphics[width=\linewidth]{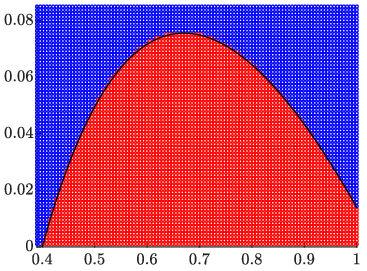}};
\node at (4,-2.5) {$u$};
\node at (-3.2,3) {$v$};
  \end{tikzpicture}
  \caption{}
\end{subfigure}
    \caption{Blue dots represent initial conditions whose corresponding orbits tend to $\mathbf{x_1}$ as $t\to +\infty$, red dots represent initial conditions whose corresponding orbits tend to the stable limit cycle contained between the $u$-axis and the curve $v=\alpha(u)$ (black curve). Values of the parameters: $l=0.4$, $m=0.67$, $\gamma=1$, and (a) $\varepsilon=0.02$ or (b) $\varepsilon=0.01$. Notice that, as $\varepsilon$ decreases, the curve $\alpha$ approximates better the border between the corresponding basins of attraction. \label{fig:test_alpha}} 
\end{figure}
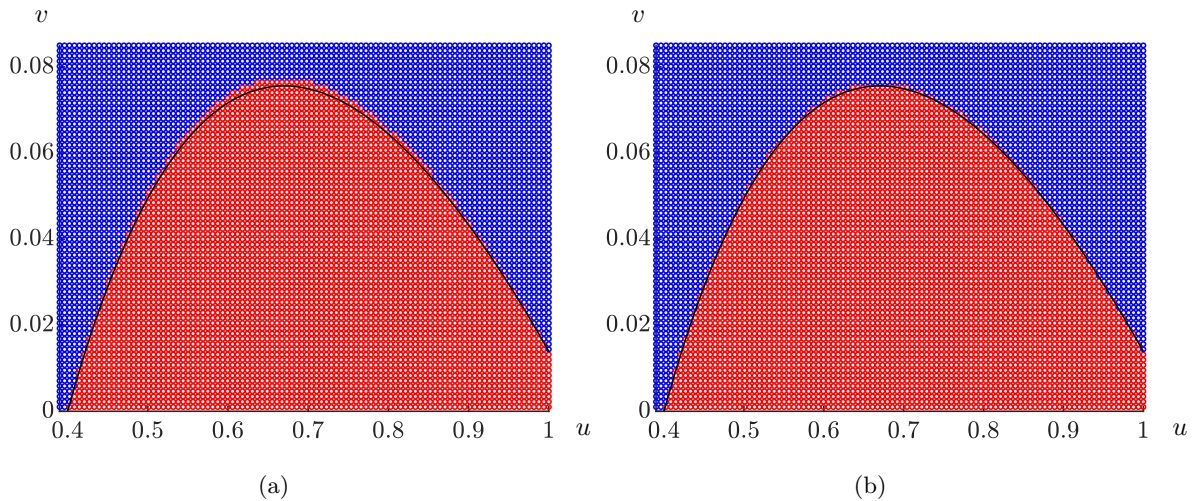

\subsection{Numerical bifurcation analysis: the (heteroclinic) cycle} \label{section4_2} 

In order to numerically showcase the creation, evolution, and collapse of the locally stable limit cycle of system \eqref{eq2}, we performed a numerical bifurcation analysis through the use of MATCONT \cite{dhooge2003matcont}. Due to the remarkable stiffness of system \eqref{eq2} even for not too small values of $\varepsilon$ (here, $\varepsilon=0.02$), we actually performed our bifurcation analysis on system \eqref{eq23}, which is considerably easier to study numerically. Indeed, \correz{for system \eqref{eq2}, the values of $v$ can be shown to be exponentially small in $\varepsilon$} (i.e.\correz{,} of the order of $\exp(-K/\varepsilon)$, for some $0<K\in \mathcal{O}(1)$), as an orbit travels through the slow flow close to $v=0$. This can be observed in Figure \ref{fig:bifurc} by noticing that $w$, from system \eqref{eq23}, reaches values smaller than $-400$, which corresponds to $v \le \exp(-400)$. The limit cycle is represented in purple, consistently with Figure \ref{fig:sketch3}. For the bifurcation analysis, we chose, consistently with Section \ref{section4_1} and Figure \ref{fig:test_alpha}, $l=0.4$, $\gamma=1$, and $\varepsilon=0.02$. Our \correz{analytical results} (see Section \ref{section3_4_3}) predicted that, as $\varepsilon\to 0$, the stable limit cycle collapses on the heteroclinic cycle for 
\begin{equation*}
    m=\dfrac{l-1}{\log l} \approx 0.6548, 
\end{equation*}
and MATCONT detects a \emph{cycle limit point} (LPC) at $m\approx 0.6556$, in good agreement with our prediction, considering the value of $\varepsilon=0.02$. Close to such value of $m$, the minimum and maximum values of $u$ of the locally stable limit cycle approach $0.4=l$ and $1$, respectively, corresponding to two equilibria connected by an heteroclinic orbit on the set $v=0$; this confirms that the limit cycle approaches the heteroclinic cycle we described above.

\begin{figure}[h!]
    \centering
\begin{subfigure}{.45\textwidth}
  \centering   
\begin{tikzpicture}
 \node at (0,0) {\includegraphics[width=\linewidth]{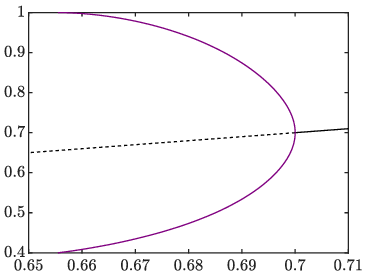}};
\node at (4,-2.5) {$m$};
\node at (-3.2,3) {$u$};
  \end{tikzpicture}
  \caption{}
\end{subfigure}\hspace{0.5cm}%
\begin{subfigure}{.46\textwidth}
  \centering
\begin{tikzpicture}
 \node at (0,0) {\includegraphics[width=\linewidth]{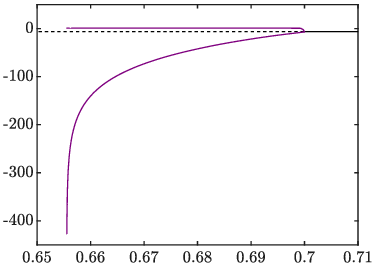}};
\node at (4,-2.5) {$m$};
\node at (-3.2,3) {$w$};
  \end{tikzpicture}
  \caption{}
\end{subfigure}
    \caption{Plot of the maximum and minimum limit cycle values of $u$ and $w$ of system \eqref{eq23} as $m$ varies. Values of the parameters: $l=0.4$, $\gamma=1$, and $\varepsilon=0.02$. For $m\in [0.7,0.71]$ the limit cycle does not exist, and the equilibrium $(\Bar{u}, \log \Bar{v})$, which corresponds to $\mathbf{x_4}$, is locally asymptotically stable (solid black line). For $m\in (\tilde{m}(l,\gamma,\varepsilon),0.7)$ the system exhibits a locally stable limit cycle around such equilibrium point (now unstable; dashed black line), which collapses on a heteroclinic cycle between $\mathbf{x_2}$ and $\mathbf{x_3}$ for $m\approx 0.6556$.\label{fig:bifurc}} 
\end{figure}
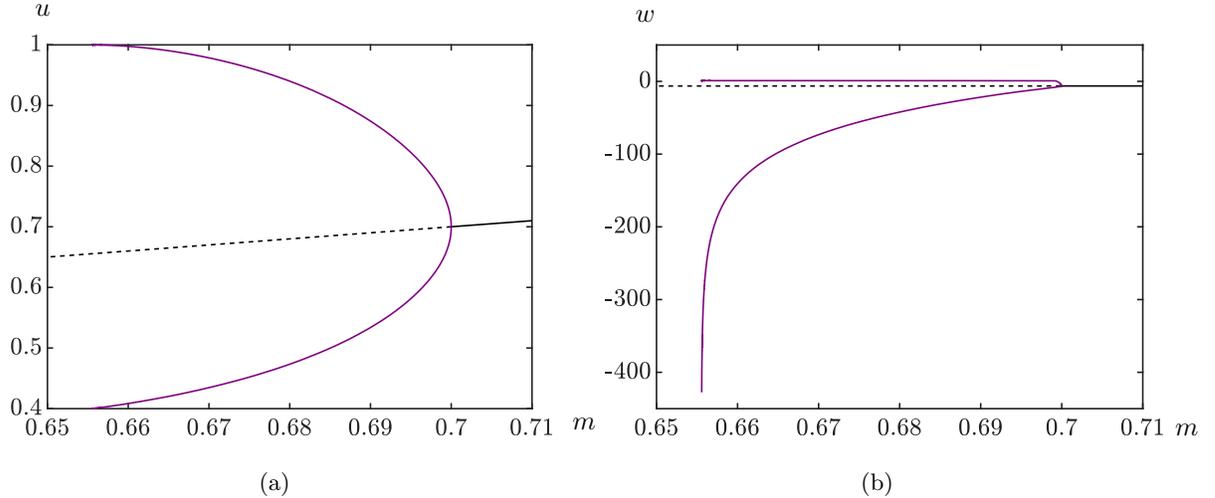

\subsection{The Poincaré map $\Pi_1 \circ \Pi_2$} \label{section4_3}

Assume that $\Tilde{m} < m < \Bar{m}$; in this case, the stable limit cycle around $\mathbf{x_4}$ exists and part of it is $\mathcal{O}(\varepsilon^2)$-close to $\mathcal{C}_0$. This means that the orbit that evolves on such cycle enters the slow flow for $u=u_\infty^\varepsilon \in (l,m)$ and exits from it for $u=u_E^\varepsilon \in (m,1)$. In the limit $\varepsilon \to 0$, $u_\infty$ is the fixed point of the map $\Pi_1 \circ \Pi_2$ while $u_E=\Pi_2(u_\infty)$. Figure \ref{fig_6_A} shows such fixed point assuming $l=0.4$, $\gamma=1$, and varying $m$ between $(l-1)/\log l$ and $(l+1)/2$. Since the fixed point exists for all these values of $m$, this implies that $\Bar{m} \to {\frac{l+1}{2}}^-$ as $\varepsilon \to 0$. Notice that $u_\infty \to l^+$ as $m \to {\frac{l-1}{\log l}}^+$, which is in line with the fact that for such value of $m$ we observe the heteroclinic orbit between $\mathbf{x_2}$ and $\mathbf{x_3}$. \correz{We summarize in Figure \ref{fig:bifurc_nuova} the possible behaviors of system \eqref{eq2}, showing that the system corresponding to a small fixed $\varepsilon$ is not topologically equivalent to the one obtained in the limit $\varepsilon \to 0$ because $\Bar{m} \to {\frac{l+1}{2}}^-$ as $\varepsilon \to 0$.}

We would like to compare the values of $u_\infty$ presented in Figure \ref{fig_6_A}, which are derived in the limit $\varepsilon \to 0$, with the values attained by $u$ as the orbit that describes the stable limit cycle enters the slow flow, for some $0<\varepsilon\ll 1$. Denote such value as $u_\infty^\varepsilon$, even if we are not able to derive it precisely since it is not possible to exactly tell in which point an orbit enters the slow flow, in Remark \ref{rmk1} we showed that an (at worst) $\mathcal{O}(\varepsilon |\log \varepsilon|)$-approximation of it is given by $u_*$ such that the point $(u_*, v_*)$ satisfies $u_*<m$, $v_*=\beta(u_*)$ \eqref{isocline}. Note that $u_*$ corresponds to the minimum value attained by $u$ throughout the stable limit cycle. Hence, it suffices to choose as initial conditions a point inside the stable limit cycle, simulate the evolution of the orbit for long times, and finally take as approximation of $u_\infty^\varepsilon$ the minimum value attained by $u$. Figure \ref{fig_6_B} reports the results of this test for different values of $m$ and as $\varepsilon$ decreases (we set $l=0.4$ and $\gamma=1$). As expected, the difference between $u_\infty$ and $u_\infty^\varepsilon$ decreases as $\varepsilon$ decreases. For values of $m$ away from $(l-1)/\log l\approx 0.6548$, the behaviors of the curves are similar. On the other hand, for values of $m$ closer to this threshold quantity, the behaviors are slightly different. In particular, for $m=0.66$ the curve is not monotone. This is a consequence of these values of $m$ being not sufficiently far from $(l-1)/\log l$ and hence some unpredictable behavior arise if $\varepsilon$ is too large. Indeed, the cases $m=0.6575$ and $m=0.66$ are the only ones where for some values of $\varepsilon$ we obtained $u_*<u_\infty$. Recall that our whole analysis hinges on the requirement that $\varepsilon$ is small enough, hence for $\varepsilon$ larger than a case-dependent threshold, we may not be able to properly predict the behavior of the system.

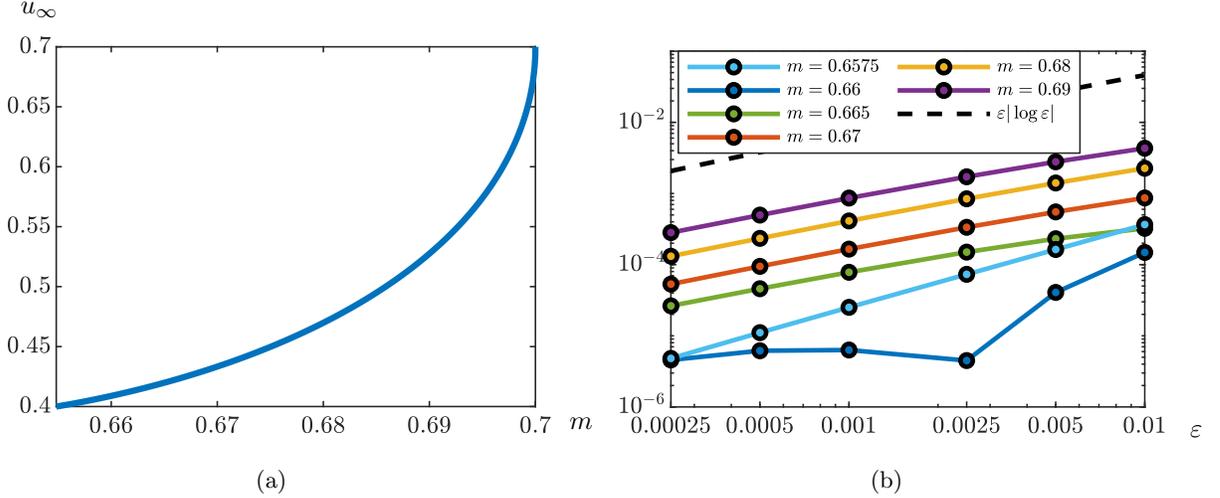
\begin{figure}[h!]
     \centering
\begin{subfigure}{.45\textwidth}
  \centering   
\begin{tikzpicture}
 \node at (0,0) {\includegraphics[width=\linewidth]{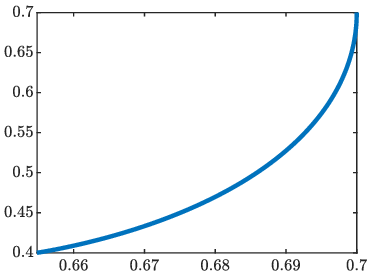} };
\node at (4,-2.5) {$m$};
\node at (-3.2,3) {$u_\infty$};
  \end{tikzpicture}
  \caption{} \label{fig_6_A}
\end{subfigure}\hspace{0.75cm}%
\begin{subfigure}{.45\textwidth}
  \centering
\begin{tikzpicture}
 \node at (0,0) {\includegraphics[width=\linewidth]{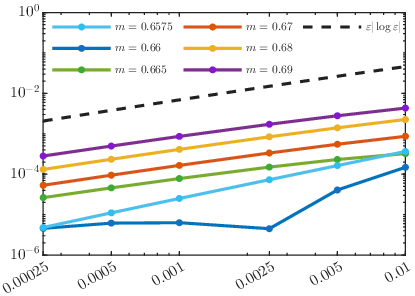}};
\node at (3.9,-2.15) {$\varepsilon$};
  \end{tikzpicture}
  \caption{} \label{fig_6_B}
\end{subfigure}
    \caption{Values of the parameters: $l=0.4$ and $\gamma=1$. (a) Fixed point $u_\infty$ of the map $\Pi_1 \circ \Pi_2$ in the limit $\varepsilon \to 0$ as $m$ varies between $(l-1)/\log l \approx 0.6548$ and $(l+1)/2=0.7$. (b) $|u_*-u_\infty|$, where $u_*$ is derived according to Section \ref{section4_3}, for different values of $m$ and as $\varepsilon$ decreases. Recall that $u_*$ is at worst an $\mathcal{O}(\varepsilon |\log \varepsilon|)$-approximation (dashed black line) of $u_\infty^\varepsilon$.\label{fig_6}} 
\end{figure}

\begin{figure}[h!]
\centering
\begin{subfigure}{.45\textwidth}
  \centering
  \begin{tikzpicture}
 \node at (0,0) {\includegraphics[width=.9\linewidth]{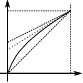}};
 \node at (0,4) {Fixed $\varepsilon \ll 1$};
 \node at (-1,3) {Case 6};
\node at (0.1,-1.6) {Case 1};
\node[rotate=45] at (-2.1,-0.7) {Case 4};
\node[rotate=30] at (-1.75,0.2) {Case 4 bis};
\node[rotate=45] at (-0.6,0) {Case 2};
\node at (3,-3) {$l$};
\node at (-3.05,3.25) {$m$};
\node at (-3,2.5) {$1$};
\node at (-3,0) {$\dfrac{1}{2}$};
\node at (2.4,-2.9) {$1$};
\node at (-1.5,2) {Case 5};
  \end{tikzpicture}
  \caption{}
  \label{fig:sketch1_nuova}
\end{subfigure}\hfill
\begin{subfigure}{.45\textwidth}
  \centering
 \begin{tikzpicture}
 \node at (0,0) {\includegraphics[width=.9\linewidth]{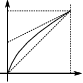}};
 \node at (-1,3) {Case 6};
 \node at (0,4) {Singular limit $\varepsilon \to 0$};
\node at (0.1,-1.6) {Case 1};
\node[rotate=45] at (-1.8,-0.1) {Case 4};
\node[rotate=45] at (-0.6,0) {Case 2};
\node at (3,-3) {$l$};
\node at (-3.05,3.25) {$m$};
\node at (-3,2.5) {$1$};
\node at (-3,0) {$\dfrac{1}{2}$};
\node at (2.4,-2.9) {$1$};
\node at (-1.5,2) {Case 5};
  \end{tikzpicture}
  \caption{}
  \label{fig:sketch2_nuova}
\end{subfigure}
\caption{\correz{Bifurcation diagram for $l$ and $m$ assuming $\gamma$ fixed and (a) $\varepsilon \ll 1$, (b) $\varepsilon \to 0$. The six Cases refer to the corresponding situations depicted in Figure \ref{fig:sketch}. Note that Case 3 (Figure \ref{fig:sketch4}) corresponds to the solid black curve between Cases 2 and 4. Scenario (a) is not topologically equivalent to scenario (b). Indeed, when $\varepsilon \ll 1$, Case 4 can be divided into two subcases: if $m<\Bar{m}$, then the cycle enters the slow flow (like in Figure \ref{fig:sketch3}; Case 4), while if $m>\Bar{m}$ it does not (Case 4 bis). On the other hand, if $\varepsilon \to 0$ then, $\Bar{m} \to (l+1)/2$ and the cycle always enters the slow flow (Case 4). Moreover, recall that $\tilde{m} \to (l-1)/\log l$ as $\varepsilon \to 0$, hence $m=(l-1)/\log l$ is the equation of the solid black curve representing Case 3 in (b).} \label{fig:bifurc_nuova}}\end{figure}

\section{Interpretation of the analytical results} \label{section5} 

We conclude the analysis of the multiple-timescale Bazykin-Berezovskaya model \eqref{eq1} with a discussion regarding the real-world interpretation of the results that we obtained. In particular, we want to focus on under what conditions the attractors of the orbits change and on how to maximize the production $P$ over large times. We will initially give an interpretation of the results in an economic setting. Finally, we will conclude with some remarks concerning the application of the model to a predator-prey scenario (recall that the original Bazykin-Berezovskaya model \cite{4,bazykin1979allee} was introduced to describe this type of biological dynamics). 

\subsection{Economic scenario} \label{section5_1}

As we showed in the previous sections, the nature of the attractors of system \eqref{eq2} is entirely determined by the relationship between the parameters $l$ and $m$. The parameter $l$ represents a threshold that distinguishes between regimes in which resources naturally grow and those in which they naturally decline. In contrast, the parameter $m$ captures the balance between the tendency of production to decrease, represented by the coefficient $d$, and its maximum possible growth, given by the maximum possible quantity of resources $M$ and the efficiency of their utilization $b$. In the context of an economy, the goal is to maximize long-term production while ensuring that resources do not get depleted. Since the parameters $e$, $b$, and $d$ in system \eqref{eq1} are the only ones that society can potentially control, we focus our attention on their influence. In terms of model \eqref{eq2}, the most favorable scenarios occur when trajectories are attracted either to the equilibrium point $\mathbf{x_4}$ or to the stable limit cycle surrounding it. Both outcomes require that $m > l$, though not excessively so. 

When $m < l$, which corresponds in the original system \eqref{eq1} to the condition $d < bL$, the decrease in production $P$ is smaller than the potential increase it would receive in the worst-case scenario where resources do not decline ($R = L$). This scenario leads to an initial rise in production as long as $R > d/b$, followed by a collapse of both $R$ and $P$ towards zero. In other words, overexploitation of resources leads to short-term gains followed by long-term collapse, as the environment fails to regenerate the consumed resources. 

If instead $m > 1$, then the decay rate of $P$ exceeds its maximum potential growth, even under optimal resource conditions ($R = M$), which inevitably leads to $P \to 0$. Two outcomes are then possible: if, at the moment production collapses, the remaining resources are sufficiently high ($R > L$), then $R$ will grow towards $M$; otherwise, $R$ will also decay to zero. These two possibilities are approximately separated by the curve $v = \alpha(u)$ (see Eq. \eqref{future}), which also depends on $\gamma$. An increase in $\gamma$ favors recovery of resources, since higher $\gamma$ results in smaller negative values of $\dot{v}$ (and hence $\dot{P}$), allowing $\dot{u}$ (and hence $\dot{R}$) to increase. 

If $m \in [\tilde{m}(l,\gamma,\varepsilon), 1)$, then the system can reach a balance (represented either by a stable point or by a limit cycle) between natural resource growth and consumption by production. Again, the curve $v = \alpha(u)$ approximately separates this balanced regime from one in which both $R$ and $P$ vanish. Interestingly, if $m \in (l, \tilde{m}(l,\gamma,\varepsilon))$, then the system still collapses ($R, P \to 0$), highlighting that the condition $d > bL$ is not sufficient by itself to guarantee sustainability. Finally, note that the small parameter $\varepsilon$ does not alter the qualitative dynamics discussed above, indeed it merely introduces a second timescale to the system. However, as we will see, the equilibrium values of $R$ and $P$ do depend on $\varepsilon$ quantitatively.

Suppose that a trajectory is attracted either to the equilibrium point $\mathbf{x_4}$ or to the stable limit cycle surrounding it. In this case, we are interested in determining the average values $R_a$ and $P_a$ assumed by the resources $R$ and the production $P$ over long time periods. In particular, our focus is on the average production $P_a$, as it reflects the sustained output resulting from our resource consumption strategy. If an orbit converges towards $\mathbf{x_4}$, since $u_a=m$ and $v_a=\varepsilon (m-l)(1-m)$, then  
\begin{equation}
    R_a = \frac{d}{b}, \quad P_a = \varepsilon \frac{n M}{e} \left(\frac{d}{b}-L\right)\left(M-\frac{d}{b}\right) \in \mathcal{O}(\varepsilon). 
\end{equation}
Assume that an orbit is attracted to the stable limit cycle, we would like to compute $R_a$ and $P_a$ for the orbit describing the cycle. Denote with $T$ the period of such orbit, then 
\begin{equation*}
    0 = \frac{1}{T} \int_0^T \frac{\dot{v}(t)}{v(t)} dt = - \frac{\gamma}{T} \int_0^T (m-u(t)) \; dt = -\gamma (m - u_a), 
\end{equation*}
implying that $R_a=d/b$, like in the previous case. An explicit value of $P_a$ can not be obtained, however we can make the following observations. If the cycle does not interact with the slow flow (namely if $m \in (\Bar{m}(l,\gamma,\varepsilon), (l+1)/2)$), then $P_a \in \mathcal{O}(\varepsilon)$. If the cycle interacts with the slow flow, then the flow on it can divided into two parts: the part evolving of the slow flow for a time $T_1 \in \mathcal{O}(1/\varepsilon)$ and the part evolving outside the slow flow for a time $T_2 \in \mathcal{O}(1)$. During the slow flow, $v$ attains an average value $v_{a,1}$ of the order $\mathcal{O}(\varepsilon^2)$, while outside such flow an average value $v_{a,2}$ of order $\mathcal{O}(1)$. Hence, 
\begin{equation*}
    v_a = \frac{v_{a,1} T_1 + v_{a,2} T_2}{T_1 + T_2} \in \mathcal{O}(\varepsilon), 
\end{equation*}
which implies that $P_a \in \mathcal{O}(\varepsilon)$. 

In conclusion, there is no substantial difference, in terms of average production, between the scenario in which the orbits are attracted to $\mathbf{x_4}$ and the one in which they are attracted to a stable limit cycle around it. However, from a practical point of view, the former is more desirable, as it avoids the risk of overexploiting resources and subsequently converging towards $\mathbf{x_1}$. The latter scenario can be interpreted as one in which production must be significantly reduced for a prolonged period to allow resources to regenerate sufficiently. In this case, long phases of near-zero production alternate with short bursts of intense activity, during which resources are rapidly exploited, leading to high production levels ($P \in \mathcal{O}(1)$). This behavior is especially evident as $m \to \tilde{m}^+$, when the limit cycle approaches the heteroclinic cycle. Conversely, when an orbit converges towards $\mathbf{x_4}$ or to a stable limit cycle that does not interact with the slow flow, production remains consistently low but stable, avoiding large oscillations.

\subsection{Predator-prey scenario} \label{section5_2}

In general, the observations previously made for the economic scenario remain valid in this framework as well (although the parameters of model \eqref{eq1} would naturally take on different interpretations). However, two disadvantages arise in this context. Model \eqref{eq2} accounts for a carrying capacity for the prey population but not for the predators, which could be considered inconsistent from a modeling standpoint. Notably, this issue does not arise in the production-resource framework, since production represents an abstract quantity that can reasonably be assumed to be potentially unbounded. For the same reason, it is acceptable for production to reach extremely low (but nonzero) values, whereas the number of predators must remain biologically plausible. This distinction is particularly relevant given that, during the slow flow, we have $v \in \mathcal{O}(\varepsilon^2)$, and the geometric analysis was performed in the limit $\varepsilon \to 0$. 

\section{Conclusions} \label{section6}

In this paper, we introduced a fast-slow version of the Bazykin–Berezovskaya predator-prey model with Allee effect and we applied it within an economic framework. Assuming that the natural growth rate of resources is significantly lower than the other parameters governing the system, we fully characterized its asymptotic behavior. We distinguished between several economic scenarios, each describing a different evolution of resources and production, and we identified those in which long-term production is maximized while preventing resource depletion. Our analysis relies on techniques from Geometric Singular Perturbation Theory (GSPT), specifically exploiting Fenichel’s Theorem and the entry-exit function. Due to the slow regeneration of resources, the orbits generally interact with a production-free manifold in different ways. The most interesting case captures a slow increase in resources following a brief period of exploitation. We showed that this cycle can repeat infinitely, if the orbit is attracted to a stable limit cycle arising from a (singular) Hopf bifurcation, or finitely, after which the system settles at the stable equilibrium $\mathbf{x_4}$ representing a balance between resource regeneration and production. Notably, we were able to analytically distinguish between scenarios where orbits are attracted to a production-free equilibrium and those (described above) where sustained cycles of production is possible. This distinction had not been achieved before in the original Bazykin–Berezovskaya model (i.e., without the small parameter $0 < \varepsilon \ll 1$). 

We performed several numerical simulations to validate the approximations made during the analytical analysis. In particular, we compared the results obtained in the limit $\varepsilon \to 0$ with those derived for small but finite values of $\varepsilon$. Remarkably, we observed a good agreement even for values of $\varepsilon$ on the order of $10^{-2}$, which is relatively large given the nature of the phenomena being modeled. Due to the stiffness of the original system, we introduced a change of variables bringing it into an equivalent version to enable accurate numerical simulations. Interestingly, these simulations revealed that the system with a small fixed $\varepsilon$ is not topologically equivalent to the one obtained in the limit $\varepsilon \to 0$. Specifically, as $\varepsilon \to 0$, the threshold value $\Bar{m}(l,\gamma,\varepsilon)$, which determines whether the stable limit cycle around $\mathbf{x_4}$ interacts with the slow flow, collapses to $(l+1)/2$, which corresponds to the occurrence of the singular Hopf bifurcation. 

As outlook for future research, one might consider a kinetic derivation of the model studied in the present paper, along the lines of what was done e.g. in \cite{MR4865531}. In particular, it would be interesting to obtain a Fokker-Planck approximation of the model (see e.g. \cite{kinetic}) and compare the related equilibria with the ones studied in the present work. Moreover, in the context of perturbed ODE models, potential generalizations of the model studied in this manuscript could include either multiple resources, multiple production variable, or both. Resources might have different growth rate, possibly letting the model evolve on three timescales. Such an approach would allow us to include both slowly regenerating resources and more sustainable resources (either with a much faster growth, or renewable resources such as solar energy). 

\bigskip
\bigskip
\textbf{Acknowledgments.} The authors are members and acknowledge the support of {\it Gruppo Nazionale di Fisica Matematica} (GNFM) of {\it Istituto Nazionale di Alta Matematica} (INdAM). JB acknowledges the support of the project PRIN 2022 PNRR ``Mathematical Modelling for a Sustainable Circular Economy in Ecosystems'' (project code P2022PSMT7, CUP D53D23018960001) funded by the European Union - NextGenerationEU, PNRR-M4C2-I 1.1, and by MUR-Italian Ministry of Universities and Research.

{\footnotesize
	\bibliographystyle{unsrt}
	\bibliography{Bibliography}

\begin{thebibliography}{10}

\bibitem{4}
A.~D. Bazykin.
\newblock {\em Nonlinear dynamics of interacting populations}, volume~11.
\newblock World Scientific Publishing Co., Inc., River Edge, NJ, 1998.
\newblock \url{https://doi.org/10.1142/9789812798725}.

\bibitem{dawes2013derivation}
J.~H.~P. Dawes and M.~O. Souza.
\newblock A derivation of {H}olling's type {I}, {II} and {III} functional responses in predator--prey systems.
\newblock {\em J. Theor. Biol.}, 327:11--22, 2013.
\newblock \url{https://doi.org/10.1016/j.jtbi.2013.02.017}.

\bibitem{40}
C.~Kuehn.
\newblock {\em Multiple Time Scale Dynamics}, volume 191.
\newblock Springer, Cham, 2015.
\newblock \url{https://doi.org/10.1007/978-3-319-12316-5}.

\bibitem{23}
N.~Fenichel.
\newblock Geometric singular perturbation theory for ordinary differential equations.
\newblock {\em J. Differ. Equ.}, 31:53--98, 1979.
\newblock \url{https://doi.org/10.1016/0022-0396(79)90152-9}.

\bibitem{34}
P.~De~Maesschalck.
\newblock Smoothness of transition maps in singular perturbation problems with one fast variable.
\newblock {\em J. Differ. Equ.}, 244:1448--1466, 2008.
\newblock \url{https://doi.org/10.1016/j.jde.2007.10.023}.

\bibitem{35}
P.~De~Maesschalck and S.~Schecter.
\newblock The entry–exit function and geometric singular perturbation theory.
\newblock {\em J. Differ. Equ.}, 260:6697--6715, 2016.
\newblock \url{https://doi.org/10.1016/j.jde.2016.01.008}.

\bibitem{jardon2021geometric}
H.~Jard{\'o}n-Kojakhmetov, C.~Kuehn, A.~Pugliese, and M.~Sensi.
\newblock A geometric analysis of the {SIR}, {SIRS} and {SIRWS} epidemiological models.
\newblock {\em Nonlinear Anal. Real World Appl.}, 58:103220, 2021.
\newblock \url{https://doi.org/10.1016/j.nonrwa.2020.103220}.

\bibitem{jardon2021geometric2}
H.~Jard{\'o}n-Kojakhmetov, C.~Kuehn, A.~Pugliese, and M.~Sensi.
\newblock A geometric analysis of the {SIRS} epidemiological model on a homogeneous network.
\newblock {\em J. Math. Biol.}, 83(4):37, 2021.
\newblock \url{https://doi.org/10.1007/s00285-021-01664-5}.

\bibitem{bulai2024geometric}
I.~M. Bulai, M.~Sensi, and S.~Sottile.
\newblock A geometric analysis of the {SIRS} compartmental model with fast information and misinformation spreading.
\newblock {\em Chaos Solit. Fractals}, 185:115104, 2024.
\newblock \url{https://doi.org/10.1016/j.chaos.2024.115104}.

\bibitem{kaklamanos2024geometric}
P.~Kaklamanos, A.~Pugliese, M.~Sensi, and S.~Sottile.
\newblock A geometric analysis of the {SIRS} model with secondary infections.
\newblock {\em SIAM J. Appl. Math.}, 84(2):661--686, 2024.
\newblock \url{https://doi.org/10.1137/23M1565632}.

\bibitem{SIRS}
J.~Borsotti.
\newblock An {SIRS} model with hospitalizations: Economic impact by disease severity.
\newblock {\em Chaos Solit. Fractals}, 200:116951, 2025.
\newblock \url{https://doi.org/10.1016/j.chaos.2025.116951}.

\bibitem{rodrigues2016time}
S.~Rodrigues, M.~Desroches, M.~Krupa, J.~M. Cortes, T.~J. Sejnowski, and A.~B. Ali.
\newblock Time-coded neurotransmitter release at excitatory and inhibitory synapses.
\newblock {\em Proc. Natl. Acad. Sci.}, 113(8):E1108--E1115, 2016.
\newblock \url{https://doi.org/10.1073/pnas.1525591113}.

\bibitem{kaklamanos2023geometric}
P.~Kaklamanos, N.~Popovi{\'c}, and K.~U. Kristiansen.
\newblock Geometric singular perturbation analysis of the multiple-timescale {H}odgkin--{H}uxley equations.
\newblock {\em SIAM J. Appl. Dyn. Syst.}, 22(3):1552--1589, 2023.
\newblock \url{https://doi.org/10.1137/22M1477477}.

\bibitem{sensi2023slow}
M.~Sensi, M.~Desroches, and S.~Rodrigues.
\newblock Slow--fast dynamics in a neurotransmitter release model: {D}elayed response to a time-dependent input signal.
\newblock {\em Phys. D: Nonlinear Phenom.}, 455:133887, 2023.
\newblock \url{https://doi.org/10.1016/j.physd.2023.133887}.

\bibitem{iuorio2021influence}
A.~Iuorio and F.~Veerman.
\newblock The influence of autotoxicity on the dynamics of vegetation spots.
\newblock {\em Phys. D: Nonlinear Phenom.}, 427:133015, 2021.
\newblock \url{https://doi.org/10.1016/j.physd.2021.133015}.

\bibitem{grifo2025far}
G.~Grif{\`o}, A.~Iuorio, and F.~Veerman.
\newblock {Far-from-Equilibrium Traveling Pulses in Sloped Semiarid Environments Driven by Autotoxicity Effects}.
\newblock {\em SIAM J. Appl. Math.}, 85(1):188--209, 2025.
\newblock \url{https://doi.org/10.1137/24M1669499}.

\bibitem{gucwa2009geometric}
I.~Gucwa and P.~Szmolyan.
\newblock Geometric singular perturbation analysis of an autocatalator model.
\newblock {\em Discrete Contin. Dyn. Syst. - S}, 2(4):783--806, 2009.
\newblock \url{https://doi.org/10.3934/dcdss.2009.2.783}.

\bibitem{kuehn2015multiscale}
C.~Kuehn and P.~Szmolyan.
\newblock Multiscale geometry of the {O}lsen model and non-classical relaxation oscillations.
\newblock {\em J. Nonlinear Sci.}, 25:583--629, 2015.
\newblock \url{https://doi.org/10.1007/s00332-015-9235-z}.

\bibitem{taghvafard2021geometric}
H.~Taghvafard, H.~Jardon-Kojakhmetov, P.~Szmolyan, and M.~Cao.
\newblock Geometric analysis of {O}scillations in the {F}rzilator model.
\newblock {\em J. Math. Anal. Appl.}, 495(1):124725, 2021.
\newblock \url{https://doi.org/10.1016/j.jmaa.2020.124725}.

\bibitem{bazykin1979allee}
A.~D. Bazykin and F.~S. Berezovskaya.
\newblock Allee effect, a low critical value of population and dynamics of system ‘predator-prey’.
\newblock {\em Problems of Ecological Monitoring and Modeling of Ecosystems}, 2:161, 1979.

\bibitem{allee}
W.~C. Allee, O.~Park, A.~E. Emerson, T.~Park, and K.~P. Schmidt.
\newblock {\em Principles of animal ecology}.
\newblock W. B. Saundere Co. Ltd., 1949.
\newblock \url{https://doi.org/10.5962/bhl.title.7325}.

\bibitem{forest}
P.G. Curtis, C.M. Slay, N.L. Harris, A.~Tyukavina, and M.C. Hansen.
\newblock Classifying drivers of global forest loss.
\newblock {\em Science}, 361(6407):1108--1111, 2018.
\newblock \url{https://www.science.org/doi/abs/10.1126/science.aau3445}.

\bibitem{SORRELL20105290}
S.~Sorrell, J.~Speirs, R.~Bentley, A.~Brandt, and R.~Miller.
\newblock Global oil depletion: A review of the evidence.
\newblock {\em Energ. Policy}, 38(9):5290--5295, 2010.
\newblock \url{https://doi.org/10.1016/j.enpol.2010.04.046}.

\bibitem{wechselberger2020geometric}
M.~Wechselberger.
\newblock {\em Geometric singular perturbation theory beyond the standard form}, volume~6.
\newblock Springer, 2020.
\newblock \url{https://link.springer.com/book/10.1007/978-3-030-36399-4}.

\bibitem{ungulati}
J.~T. Tanner.
\newblock The stability and the intrinsic growth rates of prey and predator populations.
\newblock {\em Ecology}, 56:855--867, 1975.
\newblock \url{https://doi.org/10.2307/1936296}.

\bibitem{peterson2003temporal}
R.~O. Peterson, J.~A. Vucetich, R.~E. Page, and A.~Chouinard.
\newblock Temporal and spatial aspects of predator-prey dynamics.
\newblock {\em Alces}, 39:215--232, 2003.
\newblock \url{https://alcesjournal.org/index.php/alces/article/view/483}.

\bibitem{alghe}
J.~R. Meyer, S.~P. Ellner, N.~G. Hairston, L.~E. Jones, and T.~Yoshida.
\newblock Prey evolution on the time scale of predator–prey dynamics revealed by allele-specific quantitative {PCR}.
\newblock {\em Proc. Natl. Acad. Sci. USA}, 103(28):10690--10695, 2006.
\newblock \url{https://doi.org/10.1073/pnas.0600434103}.

\bibitem{andren2024numerical}
H.~Andr{\'e}n and O.~Liberg.
\newblock Numerical response of predator to prey: Dynamic interactions and population cycles in {E}urasian lynx and roe deer.
\newblock {\em Ecol. Monogr.}, 94(1):e1594, 2024.
\newblock \url{https://doi.org/10.1002/ecm.1594}.

\bibitem{Huzak2018}
R.~Huzak.
\newblock Predator–prey systems with small predator’s death rate.
\newblock {\em Electron. J. Qual. Theory Differ. Equations}, 86:1--16, 2018.
\newblock \url{https://doi.org/10.14232/ejqtde.2018.1.86}.

\bibitem{hsu2019number}
T.-H. Hsu.
\newblock Number and stability of relaxation oscillations for predator-prey systems with small death rates.
\newblock {\em SIAM J. Appl. Dyn. Syst.}, 18(1):33--67, 2019.
\newblock \url{https://doi.org/10.1137/18M1166705}.

\bibitem{ai2024relaxation}
S.~Ai and Y.~Yi.
\newblock Relaxation oscillations in predator--prey systems.
\newblock {\em J. Dyn. Differ. Equations}, 36(Suppl 1):77--104, 2024.
\newblock \url{https://doi.org/10.1007/s10884-021-09980-6}.

\bibitem{3}
G.~A.~K. van Voorn, L.~Hemerik, M.~P. Boer, and B.~W. Kooi.
\newblock Heteroclinic orbits indicate overexploitation in predator-prey systems with a strong {A}llee effect.
\newblock {\em Math. Biosci.}, 209:451--469, 2007.
\newblock \url{https://doi.org/10.1016/j.mbs.2007.02.006}.

\bibitem{2}
N.~Berglund.
\newblock Perturbation theory of dynamical systems.
\newblock {\em \texttt{arXiv:math/0111178}}, 2001.
\newblock \url{https://doi.org/10.48550/arXiv.math/0111178}.

\bibitem{hsu2017bifurcation}
T.-H. Hsu.
\newblock On bifurcation delay: {A}n alternative approach using geometric singular perturbation theory.
\newblock {\em J. Differ. Equations}, 262(3):1617--1630, 2017.
\newblock \url{https://doi.org/10.1016/j.jde.2016.10.022}.

\bibitem{della2024geometric}
R.~Della~Marca, A.~d’Onofrio, M.~Sensi, and S.~Sottile.
\newblock A geometric analysis of the impact of large but finite switching rates on vaccination evolutionary games.
\newblock {\em Nonlinear Anal. Real World Appl.}, 75:103986, 2024.
\newblock \url{https://doi.org/10.1016/j.nonrwa.2023.103986}.

\bibitem{dhooge2003matcont}
A.~Dhooge, W.~Govaerts, and Y.~A. Kuznetsov.
\newblock {MATCONT}: a {MATLAB} package for numerical bifurcation analysis of {ODE}s.
\newblock {\em ACM Trans. Math. Softw.}, 29(2):141--164, 2003.
\newblock \url{https://doi.org/10.1145/779359.779362}.

\bibitem{MR4865531}
G.~Toscani and M.~Zanella.
\newblock On a kinetic description of {L}otka-{V}olterra dynamics.
\newblock {\em Riv. Math. Univ. Parma (N.S.)}, 15(1):61--77, 2024.
\newblock \url{https://doi.org/10.48550/arXiv.2302.14573}.

\bibitem{kinetic}
A.~Bondesan, M.~Menale, G.~Toscani, and M.~Zanella.
\newblock {L}otka-{V}olterra-type kinetic equations for competing species.
\newblock {\em Nonlinearity}, 38:075026, 2025.
\newblock \url{https://doi.org/10.1088/1361-6544/addfa1}.

\end{thebibliography}
}

\end{document}